\documentclass[ejs,preprint,twoside,numbers,natbib]{imsart}
\usepackage{amsthm,amsmath,amssymb}
\usepackage{dsfont,hyperref}%\usepackage{amthbh}
\doi{10.1214/19-EJS1545}% Updated by VTEXPTS2LaTeX.exe, 29.03.2019 11:30
\pubyear{2019}
\volume{13}
\issue{1}
\firstpage{1329}
\lastpage{1358}

\startlocaldefs
\let\mathbh\mathds
\newtheorem{thm}{Theorem}
\newtheorem{lem}[thm]{Lemma}
\newtheorem{prop}[thm]{Proposition}
\theoremstyle{definition}

\newtheorem{defin}[thm]{Definition}

\endlocaldefs

\begin{document}

\begin{frontmatter}
\title{The nonparametric LAN expansion for discretely observed diffusions}
\runtitle{LAN expansion for discretely observed diffusions}

%\title{\thanksref{t1}}
%\thankstext{t1}{}
%%% For Discussions:
% In the main article:
%\relateddois{t1}{Discussed in
%\relateddoi[ms=EJSXXXX]{Related item:}{10.1214/XX-EJSXXXX},
%\relateddoi[ms=EJSXXXX]{Related item:}{10.1214/XX-EJSXXXX} and
%\relateddoi[ms=EJSXXXX]{Related item:}{10.1214/XX-EJSXXXX}; rejoinder at
%\relateddoi[ms=EJSXXXXREJ]{Related item:}{10.1214/XX-EJSXXXXREJ}.}
% In the discussion article:
%\relateddois{t1}{Main article
%\relateddoi[ms=EJSXXXX]{Related item:}{10.1214/XX-EJSXXX}.}

 \begin{aug}
%% Corresponding author:
% Sven Wang - sven.wang@statslab.cam.ac.uk% Updated by VTEXPTS2LaTeX.exe, 29.03.2019 11:30
%\author{\fnms{} \snm{}\corref{}\ead[label=e1]{}\thanksref{a1}}
%\address{\printead{e1}}
%\thankstext{a1}{}
%\end{aug}
%\medskip\textbf{\and}
%\begin{aug}
%\author{\fnms{} \snm{}\ead[label=e2]{}}
%\address{\printead{e2}}
\author{\fnms{Sven} \snm{Wang}\ead[label=e1]{sven.wang@statslab.cam.ac.uk}}
\address{Centre for Mathematical Sciences, Wilberforce Road \\
Cambridge CB3 0WB, United Kingdom\\
\printead{e1}}
 \end{aug}

 \runauthor{S. Wang}

\begin{abstract}
Consider a scalar reflected diffusion $(X_{t}:t\geq 0)$, where the
unknown drift function $b$ is modelled nonparametrically. We show that
in the low frequency sampling case, when the sample consists of
$(X_{0},X_{\Delta },...,X_{n\Delta })$ for some fixed sampling distance
$\Delta >0$, the model satisfies the local asymptotic normality (LAN)
property, assuming that $b$ satisfies some mild regularity assumptions.
This is established by using the connections of diffusion processes to
elliptic and parabolic PDEs. The key tools used are regularity estimates
for certain parabolic PDEs as well as a detailed analysis of the
spectral properties of the elliptic differential operator related to
$(X_{t}:t\geq 0)$\vadjust{\vspace{-2pt}}.
\end{abstract}

 \vspace{-2pt}
%%Upper case for the first keyword
\vspace{-2pt}

% history:
\received{\smonth{3} \syear{2018}}% Updated by VTEXPTS2LaTeX.exe, 29.03.2019 11:30

\end{frontmatter}
\maketitle

\vspace{-4pt}
\section{Introduction}
\vspace{-2pt}

Consider a scalar diffusion, described by a stochastic differential
equation (SDE)
\begin{equation*}
dX_{t}=b(X_{t})dt+\sqrt{2}dW_{t},
\qquad
t\geq 0,
\end{equation*}
where $(W_{t}:t\geq 0)$ is a standard Brownian motion and $b$ is the
unknown \textit{drift function} that is to be estimated. We investigate
the so-called \textit{low frequency} observation scheme, where the data
consists of states
\begin{equation}
\label{eqsample}
X^{(n)}=(X_{0},X_{\Delta }...,X_{n\Delta })
\end{equation}
of one sample path of $(X_{t}:t\ge 0)$, where $\Delta >0$ is the
\textit{fixed} time difference between measurements. To ensure
ergodicity and to limit technical difficulties, we follow
\cite{ghr} and \cite{ns} and consider a version of the model where
the diffusion takes values on $[0,1]$ with reflection at the boundary
points $\{0,1\}$, see Section \ref{sec-model} for the precise
definition.

The nonparametric estimation of the coefficients of a diffusion process
has attracted a great deal of attention in the past. For the
low-frequency sampling scheme (\ref{eqsample}), Gobet, Hoffmann and
Reiss \cite{ghr} determined the minimax rate of estimation for
both the drift and diffusion coefficient and also devised a spectral
estimation method which achieves this rate. Thereafter, Nickl and
S\"{o}hl \cite{ns} proved that the Bayesian posterior distribution
contracts at the minimax rate, giving a frequentist justification for
the use of Bayesian methods. In other sampling schemes, various methods
have been studied, see e.g. \cite{hoffmann1} for a frequentist
approach, \cite{ns15,ns16,ns33,ABR18,NR18} for recent posterior
consistency and contraction rate results for Bayesian methods as well
as \cite{ns19,ns29} for MCMC methodology for the computation of
the Bayesian posterior.

However, often one desires a more detailed understanding of the
performance of both frequentist and Bayesian methods, e.g. by
establishing semi-parametric efficiency bounds or by proving a
nonparametric Bernstein-von Mises theorem (BvM), which would give a
frequentist justification for the use of Bayesian credible sets as
confidence sets (see \cite{nicklgine}, Chapter 7.3). Nonparametric
BvMs have been explored in the papers \cite{cn1,cn2} and have
recently been proven for a number of statistical inverse problems
\cite{n17,ns17,monard}, by Nickl and co-authors. In a diffusion model
with continuous observations $(X_{t}:t\le T)$, Nickl and Ray
\cite{NR18} recently proved a nonparametric BvM for estimating the drift
$b$.

To order to achieve such a detailed understanding, a key step lies in
studying the local information geometry of the parameter space, which
in terms of semiparametric efficiency theory (see e.g.
\cite{vdv}, Chapter 25) involves finding the LAN expansion and the
corresponding (Fisher) information operator. This in turn determines the
Cram\'{e}r-Rao lower bound for estimating a certain class of functionals
of the parameter of interest. While in the Gaussian white noise model
with direct observations, the LAN expansion of the log-likelihood ratio
is exact and given by the Cameron-Martin theorem, in inverse problems
proving the LAN property is often not straightforward.

In a finite-dimensional (parametric) model for multidimensional
diffusions which are sampled at high frequency, where the sample
consists of states
\begin{equation*}
X^{(n)}=(X_{0},X_{\Delta_{n}}...,X_{n\Delta_{n}})
\end{equation*}
with asymptotics such that $\Delta_{n}\to 0$ and $n\Delta_{n}\to
\infty $, the LAN property was shown by Gobet \cite{gobetlan} by
use of Malliavin calculus.

The main contribution of this paper is to prove that also with
\emph{low} frequency observations, the reflected diffusion model
satisfies the LAN property, under mild regularity assumptions on the
drift $b$. If the transition densities of the Markov chain $(X_{i
\Delta }:i\in \mathbb{N})$ are denoted by $p_{\Delta ,b}$, then the
log-likelihood of the sample (\ref{eqsample}) is approximately equal to
\begin{equation*}
\ell_{b} (X^{(n)})\approx \sum_{i=1}^{n} \log p_{\Delta ,b}(X_{(i-1)
\Delta },X_{i\Delta }),
\end{equation*}
from which one can see the necessity of two ingredients to show the LAN
expansion:
\begin{itemize}
\item
The first is a result on the differentiability of the transition
densities $b\mapsto p_{\Delta ,b}(x,y)$, which guarantees that we can
form the second-order Taylor expansion of the log-likelihood in certain
`directions' $h/\sqrt{n}$ with sufficiently good control over the
remainder. See Theorem \ref{thmderivative} for the precise statement,
where we importantly also obtain an explicit form for the first
derivative $A_{b}$, the `score operator'.
\item
The second main ingredient consists of two well known limit theorems,
the central limit theorem for martingale difference sequences
\cite{brown} and the ergodic theorem, which ensure the right limits for
the first and second order terms in the Taylor expansion respectively.
\end{itemize}
In view of this, the main work done in this paper lies in establishing
the regularity needed for $p_{\Delta ,b}(x,y)$, see Theorem
\ref{thmderivative} below. As there is no explicit formula for
$p_{\Delta ,b}(x,y)$ in terms of $b$, our approach relies on techniques
from the theory of parabolic PDE and spectral theory. We use a PDE
perturbation argument, based on the fact that the transition densities
of a diffusion process can naturally be viewed as the fundamental
solution to a related parabolic PDE.

The main difficulty in the proofs lies in the singular behaviour of
$p_{t,b}(x,y)$ as $(t,x)$ approaches $(0,y)$, which is why standard PDE
results cannot be applied directly, but only in a regularised setting.
Thus the arguments will first be carried out for any fixed
regularisation parameter $\delta >0$, where the analysis needs to be
done carefully in order to ensure that the estimates obtained are
uniform in $\delta >0$ and hence still valid in the limit $\delta
\to 0$.

In the context of a statistical inverse problem for the (elliptic)
Schr\"{o}dinger equation \cite{n17,lu17}, where the above singular
behaviour is not present, PDE perturbation arguments have previously
been used to linearize the $\log $-likelihood.

We also remark that the use of more probabilistic proof techniques like
in \cite{gobetlan} would have been conceivable, too. However, we
found the PDE approach employed here to be more naturally suited to
dealing with boundary conditions, and it avoids dealing with pathwise
properties of the diffusions by working with the transitions densities
directly, which are ultimately the objects of interest for analyzing the
likelihood.

Potential applications of the LAN expansion presented in Theorem
\ref{thmlanexpansion} include the study of semiparametric efficiency for
a certain class of functionals of $b$ which is implicitly defined by the
range of the `information operator' $A^{*}_{b}A_{b}$ (where
$A_{b}$ is the score operator (\ref{eqscoreoperator})), as well as an
infinite-dimensional Bernstein-von-Mises theorem similar to
\cite{monard,n17,NR18,ns17}. However, studying the properties of
$A^{*}_{b}A_{b}$ needed for this poses a highly non-trivial challenge
which still has to be overcome, see Section \ref{sec-appl} for a more
detailed discussion.

In Section \ref{sec-main-res}, we state and prove the LAN expansion.
Section \ref{seclocalapprox} is devoted to proving Theorem
\ref{thmderivative}. Finally, in Section \ref{secspectral}, we derive
the spectral properties of the differential operator $\mathcal{L}_{b}$
and the transition semigroup $(P_{t,b}:t\geq 0)$ needed throughout the
proofs.

\section{Main results}\label{sec-main-res}

\subsection{A reflected diffusion model}\label{sec-model}
We shall work with boundary reflected diffusions on the interval
$[0,1]$, following \cite{ghr,ns}. Consider the stochastic process
$(X_{t}:t\geq 0)$, whose evolution is described by the stochastic
differential equation (SDE)
\begin{equation}
\label{eqSDE}
dX_{t}=b(X_{t})dt+\sqrt{2}dW_{t}\;+\;dK_{t}(X),
\qquad
X_{t}\in [0,1],
~~
t\geq 0.
\end{equation}
Here $(W_{t}:t\geq 0)$ is a standard Brownian motion, $(K_{t}(X):t
\geq 0)$ is a non-anticipative finite variation process that only
changes when $X_{t}\in \{0,1\}$ and
\begin{equation*}
b:[0,1]\to \mathbb{R}
\end{equation*}
is the unknown drift function. We note that $K(X)$, which accounts for
the reflecting boundary behaviour, is part of a solution to
(\ref{eqSDE}) and is in fact given by the difference of the local times
of $X$ at $0$ and $1$.

For any integer $s\ge 0$, let $C^{s}=C^{s}((0,1))$ and
$H^{s}=H^{s}((0,1))$ denote the spaces of $s$-times continuously
differentiable functions and $s$-times weakly differentiable functions
with $L^{2}$-derivatives, respectively, endowed with the usual norms.
We also define the subspace
\begin{equation*}
C^{1}_{0}:=\{f\in C^{1}:f(0)=f(1)=0\}.
\end{equation*}
We assume throughout that for some $B<\infty $, $b$ lies in the
$C^{1}_{0}$-ball
\begin{equation}
\label{theta-def}
\Theta :=\big \{f\in C^{1}_{0}:\|f\|_{C^{1}}:=\|f\|_{\infty }+\|f'\|
_{\infty }\le B\big \}.
\end{equation}
This ensures the existence of a pathwise solution $(X_{t}:t\geq 0)$ to
(\ref{eqSDE}) which can be constructed by a reflection argument, see
e.g. Section I.{\S }23 in \cite{GS72} or \cite{ns}. For some
$\Delta >0$, which we assume to be \textit{fixed} throughout the paper,
our sample consists of measurements $X^{(n)}=(X_{0},X_{\Delta }...,X
_{n\Delta })$ of one sample path, with asymptotics $n\to \infty $.

The process $(X_{t}:t\geq 0)$ forms an ergodic Markov process with
invariant distribution $\mu_{b}=\mu $, whose Lebesgue density (which we
also denote by $\mu_{b}$) is identified by $b$ via
\begin{equation}
\label{eqinvariantmeasure}
\mu_{b}(x)=
\frac{e^{\int_{0}^{x}b(y)dy}}{\int_{0}^{1}e^{\int_{0}^{u}b(y)dy}du},
\quad x\in [0,1],
\end{equation}
see e.g. Chapter 4 in \cite{bass}. Moreover, we denote the
Lebesgue transition densities and the semigroup associated to
$(X_{t}:t\ge 0)$ by $p_{t,b}$ and $P_{t,b}$ respectively:
\begin{align}
&p_{t,b}:[0,1]^{2}\to \mathbb{R},
~~
p_{t,b}(x,y)=\mathbb{P}_{x}(X_{t}\in dy),
~~
t>0,
\label{eqtransitiondensity}\\
&P_{t,b}f (x)=\mathbb{E}_{x}[f(X_{t})]=\int_{0}^{1}p_{t,b}(x,z)f(z)dz,
~~
t>0,
~~
f\in L^{2}.
\label{eqsemigroup}
\end{align}
Here, by Proposition 9 in \cite{ns}, the transition densities are
well--defined as well as bounded above and below for each $t>0$, so that
(\ref{eqsemigroup}) is well-defined, too.

Let $\mathbb{P}_{b}$ denote the law of $(X_{i\Delta }:i\geq 0)$ on
$[0,1]^{\mathbb{N}}$. For ease of exposition, we assume throughout that
$X_{0}\sim \mu_{b}$ under $\mathbb{P}_{b}$, a common assumption (cf.
\cite{ghr,ns}) which we make due to the uniform spectral gap over
$b\in \Theta $ guaranteed by Lemma \ref{lem bounds on eigenfunctions}
below, which yields exponentially fast convergence of $X_{t}$ to
$\mu_{b}$. Then under any $\mathbb{P}_{b}$, $b\in \Theta $, the law of
$X^{(n)}$ from (\ref{eqsample}) on $[0,1]^{n+1}$ is absolutely
continuous with respect to the $n+1$-- dimensional Lebesgue measure, and
the log--density, which also constitutes the \textit{log--likelihood}
(when viewed as a function of $b$), is given by
\begin{equation}
\label{eqloglikelihood}
\log d\mathbb{P}_{b}(X^{(n)})=\log \mu_{b}(X_{0})+\sum_{i=1}^{n}
\log p_{\Delta ,b}(X_{(i-1)\Delta }, X_{i\Delta }).
\end{equation}

We note that some of the above assumptions can be relaxed at the expense
of further technicalities in the proofs: Firstly, the assumption
$X_{0}\sim \mu_{b}$ could be replaced by $X_{0}\sim \pi_{b}$ (under
$\mathbb{P}_{b}$), for any measures $\pi_{b}$ with Lebesgue densities
such that for all $b\in \Theta $, $\log d\pi_{\tilde{b}}(X_{0})-
\log d\pi_{b}(X_{0})=o_{\mathbb{P}_{b}}(1)$ as $\|\tilde{b}-b\|_{H
^{1}}\to 0$. Secondly, it is conceivable that the main Theorems
\ref{thmderivative} and \ref{thmlanexpansion} below can be generalized
to all $b\in H^{1}$ and $h\in \{f\in H^{1}:f(0)=f(1)=0 \}$, which we
shall not pursue further here, however.

\subsection{Differentiability of the transition densities}

In order to prove the LAN property, we need to differentiate the
log-likelihood (\ref{eqloglikelihood}) at any drift parameter
$b\in \Theta $, and the following theorem shows that for any
$x,y\in [0,1]$, maps of the form $b\mapsto p_{\Delta ,b}(x,y)$ are
infinitely differentiable in `directions' $h\in C^{1}_{0}$ (and in fact,
Fr\'{e}chet differentiable). For $b,h\in C^{1}_{0}$, $\eta \in
\mathbb{R}$ and $x,y\in [0,1]$, for convenience we introduce the
notation
\begin{equation*}
\Phi_{b,h,x,y}= \Phi : \mathbb{R}\to \mathbb{R},
\qquad
\Phi (\eta )=p_{\Delta ,b+\eta h}(x,y).
\end{equation*}
\begin{thm}
\label{thmderivative}
For all $b,h\in C^{1}_{0}$ and $x,y\in [0,1]$, $\Phi =\Phi_{b,h,x,y}$
is a smooth (in fact, real analytic) function on $\mathbb{R}$, and we
have
\begin{equation}
\label{eqderivative}
\Phi '(0)=\int_{0}^{\Delta }P_{\Delta -s,b}[h\partial_{1}p_{s,b}(
\cdot ,y)](x)ds.
\end{equation}
Moreover, for each integer $k\geq 1$, we have the following bound on the
$k$-th derivative of $\Phi $ at $0$:
\begin{equation*}
\sup_{b\in \Theta }\;\sup_{h\in C^{1}_{0}, \|h\|_{H^{1}}\leq 1}\;
\sup_{x,y\in [0,1]}\left| \Phi^{(k)}(0)\right| <\infty .
\end{equation*}
\end{thm}

Section \ref{seclocalapprox} is devoted to the proof of Theorem
\ref{thmderivative}.

Heuristically speaking, the right hand side of (\ref{eqderivative}) has
the form of a solution to an inhomogeneous parabolic PDE (cf.
Proposition \ref{proplunardi}), and this PDE perspective will be key in
the proofs. However, one has to be careful with such an interpretation,
as the singular `source term' $h\partial_{1}p_{b,t}(\cdot ,y)$ does not
fall within the scope of classical PDE theory. Therefore, the above
intuition needs to be made rigorous via a regularisation argument, see
Section \ref{seclocalapprox}.

\subsection{LAN expansion}\label{subseclan}
By Lemma \ref{lem bounds on transition densities}, for each
$b\in \Theta $, $p_{\Delta ,b}(\cdot ,\cdot )$ is bounded above and
below. Hence by Theorem \ref{thmderivative} and the chain rule, the
score operator is given by
\begin{equation}
\label{eqscoreoperator}
\begin{split}
A_{b}:
~~
C^{1}_{0}([0,1])\;\to \; L^{2}([0,1]\times [0,1]),
~~~~
A_{b}h(x,y)=\frac{[ \Phi_{b,h,x,y}]'(0)}{p_{\Delta , b}(x,y)}.
\end{split}
\end{equation}
For any $f,g\in L^{2}([0,1]\times [0,1])$, we also define the
corresponding `LAN inner product' and `LAN norm' as follows:
\begin{equation}
\label{lan-prod}
\begin{split}
\langle f,g\rangle_{L^{2}(p_{\Delta ,b}\mu_{b})}
&:=\int_{0}^{1} \int
_{0}^{1} f(x,y)g(x,y) \mu_{b}(x)p_{\Delta ,b}(x,y) dx dy,
\\
\langle f,g\rangle_{LAN}:=\langle
&A_{b}f,A_{b}g
\rangle_{L^{2}(p_{\Delta ,b}\mu_{b})},
~~~~~~~
\|f\|_{LAN}^{2} := \langle f,f\rangle_{LAN}.
\end{split}
\end{equation}
Here is our main result, the proof can be found in Section
\ref{sec-lan-pf}.
\begin{thm}[LAN expansion]
\label{thmlanexpansion}
For any $b,h\in C^{1}_{0}$, we have that
\begin{equation}
\label{eqlan1}
\log \frac{d\mathbb{P}_{b+h/\sqrt{n}}}{d\mathbb{P}_{b}}(X^{(n)}) = \frac{1}{
\sqrt{n}} \sum_{i=1}^{n} A_{b}h(X_{(i-1)\Delta },X_{i\Delta }) -
\frac{1}{2} \|h\|^{2}_{LAN} + o_{\mathbb{P}_{b}}(1)
\end{equation}
as $n\to \infty $ and
\begin{equation}
\label{eqlan2}
\frac{1}{\sqrt{n}} \sum_{i=1}^{n} A_{b}h(X_{(i-1)\Delta },X_{i
\Delta }) \xrightarrow{n\to \infty }^{d} N(0, \|h\|_{LAN}^{2}).
\end{equation}
\end{thm}

Note that due to the nature of the non-i.i.d. Markov chain data at hand,
$A_{b}$ necessarily needs to map into a function space of two variables,
as the overall $\log $-likelihood cannot be formed as a sum of functions
of single states of the chain, but only of increments of the chain.

\subsection{Potential statistical applications of Theorem
\ref{thmlanexpansion}}\label{sec-appl}
The LAN expansion can be used to obtain semiparametric lower bounds for
the estimation of certain linear functionals $L(b)$ for which there
exists a Riesz representer $\Psi \in C^{1}_{0}$ such that $L(\cdot )=
\langle \Psi , \cdot \rangle_{LAN}$, and can potentially further be used
to prove a non-parametric Bernstein-von-Mises theorem.

To make this more precise, we define the `information operator' (which
generalizes the Fisher information) by $I_{b}:=A_{b}^{*}A_{b}: C^{1}
_{0}\to L^{2}$, where $A_{b}$ from (\ref{eqscoreoperator}) is viewed as
a densely defined operator on $L^{2}$ with domain $C^{1}_{0}$ and
$A_{b}^{*}$ is the adjoint of $A_{b}$ with respect to the inner products
$\langle \cdot , \cdot \rangle_{L^{2}}$ and $\langle \cdot ,\cdot
\rangle_{L^{2}(p_{\Delta ,b}\mu_{b})}$. Then, for example, to study
semiparametric Cram\'{e}r-Rao lower bounds for functionals of the form
$L(b)=\langle \psi ,b\rangle_{L^{2}}$, $\psi \in L^{2}$, one needs that
there is some $\Psi \in C^{1}_{0}$ such that
\begin{equation*}
\forall w\in C^{1}_{0}:~~ \langle \psi ,w\rangle_{L^{2}}=\langle
\Psi ,w\rangle_{LAN}=\langle I_{b}\Psi , w\rangle_{L^{2}}.
\end{equation*}
Hence one needs that $\psi $ lies in the range $R(I_{b})$ of
$I_{b}$ (or at least of $R(A_{b}^{*})$), see p.372-373 in
\cite{vdv} for a detailed discussion. Assuming the injectivity of
$I_{b}$, the `optimal asymptotic variance' for estimators of
$L(b)$ is then given by
\begin{equation*}
\|\Psi \|_{LAN}^{2}=\langle A_{b}\Psi ,A_{b}\Psi
\rangle_{L^{2}(p_{\Delta ,b}\mu_{b})},
\end{equation*}
which may intuitively be understood as an `inverse Fisher information
$\langle \psi ,\break I_{b}^{-1}\psi \rangle_{L^{2}}$', in analogy to the
parametric setting.

When $R(I_{b})$ is known to contain at least a `nice' subspace of
functions, e.g. $C_{c}^{\infty }$, $I_{b}$ can be inverted on that
subspace, and if key mapping properties of $I_{b}^{-1}$ are known, then
along the lines of \cite{n17,ns17,NR18,monard}, one can further
try to prove a nonparametric BvM. This would assert the convergence of
infinite-dimensional posterior distributions to a Gaussian limit measure
$\mathcal{G}$ whose covariance is given by the LAN inner product via
$\textnormal{Cov}[G(\psi_{1}),G(\psi_{2})]=\langle \Psi_{1},\Psi_{2}
\rangle_{LAN}$, cf. (28) in \cite{n17}.

The identification of $R(I_{b})$ in the present case of diffusions
sampled at low frequency, as well as the study of mapping properties of
$I_{b}$, remain challenging open problems.

\subsection{Proof of the LAN expansion}\label{sec-lan-pf}
We now give the proof of Theorem \ref{thmlanexpansion}, assuming the
validity of Theorem \ref{thmderivative} which is proven in Section
\ref{seclocalapprox} below. Besides Theorem \ref{thmderivative}, the
other key ingredient for Theorem \ref{thmlanexpansion} is the following
CLT for martingale difference sequences. It is due to Brown (building
on ideas of Billingsley and L\'{e}vy) and follows immediately from the
special case $t=1$ in Theorem 2 of \cite{brown}.
\begin{prop}[cf. \cite{brown}]
\label{propmgclt}
Suppose $(\Omega , \mathcal{F},(\mathcal{F}_{n}:n\geq 0),\mathbb{P})$
is a filtered probability space and let $(M_{n}:n\in \mathbb{N})$ be a
$\mathcal{F}_{n}$-martingale with $M_{0}=0$. For $n\ge 1$, define the
increments $Y_{n}:=M_{n}-M_{n-1}$ and let
\begin{equation*}
\sigma_{n}^{2}:=\mathbb{E} \left[ Y_{n}^{2} \middle| \mathcal{F}_{n-1}
\right] ,
\qquad
V_{n}^{2}:=\sum_{i=1}^{n} \sigma_{i}^{2},
\qquad
s_{n}^{2}:=\mathbb{E}[V_{n}^{2}].
\end{equation*}
Suppose that $V_{n}^{2}s_{n}^{-2}\xrightarrow{n\to \infty } 1$ in
probability and that for all $\epsilon >0$,
\begin{equation}
\label{eqlindebergcond}
\frac{1}{s_{n}^{2}}\sum_{i=1}^{n}\mathbb{E}\left[ Y_{i}^{2}\mathbh{1}
\{|Y_{i}|\geq \epsilon s_{n}\}\right] \xrightarrow{n\to \infty } 0
\end{equation}
in probability. Then, as $n\to \infty $, we have
\begin{equation*}
M_{n}/s_{n}\xrightarrow{d}\mathcal{N}(0,1).
\end{equation*}
\end{prop}
\begin{proof}[Proof of Theorem \ref{thmlanexpansion}]
Fix $b,h\in C^{1}_{0}$. Due to the spectral gap of the generator
$\mathcal{L}_{b}$ (see Lemma \ref{lem bounds on eigenfunctions}), the
Markov chain $(X_{n\Delta }:n\in \mathbb{N})$ originating from the
diffusion $(\ref{eqSDE})$ with initial distribution $X_{0}\sim \mu
_{b}$, is stationary and geometrically ergodic -- we will use this fact
repeatedly.

For notational convenience, we write
\begin{equation*}
f(\eta ,x,y)=\log p_{\Delta ,b+\eta h}(x,y),
\qquad
g(\eta ,x,y)=p_{\Delta ,b+\eta h}(x,y).
\end{equation*}
By Theorem \ref{thmderivative}, $f$ is smooth in $\eta $ on a
neighbourhood of $0$, and for some $C< \infty $, the second order Taylor
remainder satisfies
\begin{equation}
\label{eqerrorcontrol}
R_{f}(\eta ):=\sup_{x,y\in [0,1]}|f(\eta ,x,y)-f(0,x,y)-\eta
\partial_{\eta }f(0,x,y)-\frac{\eta^{2}}{2}\partial_{\eta }^{2} f(0,x,y)|
\leq C|\eta |^{3}.
\end{equation}
Thus, Taylor-expanding the log-likelihood (\ref{eqloglikelihood}) in
direction $h/\sqrt{n}$ yields that
\begin{equation}
\label{eqlikelihoodexpansion}
\begin{split}
\log
& \frac{d\mathbb{P}^{n}_{b+h/\sqrt{n}}}{d\mathbb{P}^{n}_{b}}(X
_{0},...,X_{n_{\Delta }})
\\
&=\big(\log \mu_{b+h/\sqrt{n}}(X_{0})-\log \mu_{b}(X_{0})\big)+\frac{1}{
\sqrt{n}}\sum_{i=1}^{n} \partial_{\eta }f(0,X_{(i-1)\Delta },X_{i
\Delta })
\\
&
\;\;
+\frac{1}{2n}\sum_{i=1}^{n} \partial_{\eta }^{2}f(0,X_{(i-1)\Delta },X
_{i\Delta })+D_{n}
\\
&=:A_{n}+B_{n}+C_{n} +D_{n}.
\end{split}
\end{equation}

For the remainder term $D_{n}$, we immediately see from
(\ref{eqerrorcontrol}) that $|D_{n}|\leq nR_{f}(n^{-1/2})$, whence
$D_{n}=o_{\mathbb{P}_{b}}(1)$ as $n\to \infty $.

For $C_{n}$, observe that the function $\partial_{\eta }^{2}f(0,
\cdot ,\cdot )$ is bounded by Theorem \ref{thmderivative}, such that the
almost sure ergodic theorem yields that
\begin{equation*}
C_{n}\xrightarrow{n\to \infty } \frac{1}{2}\mathbb{E}_{b}[
\partial_{\eta }^{2}f(0,X_{0},X_{\Delta })]\quad \text{ a.s.},
\end{equation*}
where $\mathbb{E}_{b}$ denotes the expectation with respect to
$\mathbb{P}_{b}$. Moreover, we have
\begin{equation*}
\begin{split}
\mathbb{\partial }_{\eta }^{2}f(0,X_{0},X_{\Delta })=\frac{\partial
^{2}_{\eta }g(0,X_{0},X_{\Delta })}{g(0,X_{0},X_{\Delta })}-\big(
\partial_{\eta }f(0,X_{0},X_{\Delta })\big)^{2}=: I+II,
\end{split}
\end{equation*}
and by interchanging differentiation and integration (which is possible
by Theorem \ref{thmderivative}), we see that
\begin{equation*}
\mathbb{E}[I]=\int_{0}^{1}\int_{0}^{1} \partial_{\eta }^{2}g(0,x,y)\mu
_{b}(x)dxdy=0,
\end{equation*}
and hence $\mathbb{E}_{b}[\partial_{\eta }^{2}f(0,X_{0},X_{\Delta })]=-
\langle A_{b}h,A_{b}h\rangle_{L^{2}(p_{\Delta ,b}\mu_{b})}=-\|h\|^{2}
_{LAN}.$

We next treat $B_{n}$. Let $(\mathcal{F}_{n}:n\geq 0)$ denote the
natural filtration of $(X_{\Delta n}:n\geq 0)$. In view of Proposition
\ref{propmgclt}, let us write
\begin{equation*}
\begin{split}
Y_{n}
&=\;\partial_{\eta }f(0,X_{(n-1)\Delta },X_{n\Delta }),
~~~~
M_{n}:=\sqrt{n} B_{n}=\sum_{i=1}^{n}Y_{n},
\\
\sigma^{2}_{n}
&=\mathbb{E}\left[ Y_{n}^{2} \middle| X_{(n-1)\Delta }\right] ,
\qquad
V_{n}^{2}=\sum_{i=1}^{n} \sigma_{i}^{2},
\qquad
s_{n}^{2}=\mathbb{E}[V_{n}^{2}].
\end{split}
\end{equation*}
Then, using dominated convergence and the Markov property, we see that
$M_{0}=0$ and that $(M_{n}:n\ge 0)$ is a martingale:
\begin{equation*}
\begin{split}
\mathbb{E}[Y_{n}|\mathcal{F}_{n-1}]
&=\int_{0}^{1}\partial_{\eta }f(0,X
_{(n-1)\Delta },y)p_{\Delta ,b}(X_{(n-1)\Delta },y)dy
\\
&=\int_{0}^{1}\partial_{\eta }g(0,X_{(n-1)\Delta },y)dy
\\
&=\partial_{\eta }\int_{0}^{1}p_{\Delta ,b+\eta h}(X_{(n-1)\Delta },y)dy~
\Big|_{\eta =0}=0.
\end{split}
\end{equation*}
Moreover, we have that $\sigma^{2}_{n}=\tilde{\sigma }^{2}(X_{(n-1)
\Delta })$ for some bounded measurable function $\tilde{\sigma }^{2}:[0,1]
\to [0,\infty )$ and by the stationarity of $(X_{i\Delta }:i\geq 0)$,
we have $s_{n}^{2}=n \mathbb{E}_{b}[\tilde{\sigma }^{2}(X_{0})]=n\|h
\|_{LAN}^{2}$, whence the ergodic theorem yields that $\mathbb{P}_{b}$--
a.s.,
\begin{equation*}
V_{n}^{2}s_{n}^{-2}=\frac{1}{n\|h\|_{LAN}^{2}}\sum_{i=1}^{n}
\tilde{\sigma }^{2}(X_{(n-1)\Delta }) \xrightarrow{n\to \infty } \|h
\|_{LAN}^{-2}\mathbb{E}_{b}[(\partial_{\eta }f(0,X_{0},X_{1}))^{2}]=1.
\end{equation*}
Lastly, as the $Y_{i}$'s are bounded random variables, the condition
(\ref{eqlindebergcond}) is fulfilled. Hence Proposition
\ref{propmgclt} yields that
$B_{n}\to^{d}\mathcal{N}(0,\|h\|^{2}_{LAN})$.

Finally, we observe that the term $A_{n}$ in
(\ref{eqlikelihoodexpansion}) from the invariant measure is of order
$o_{\mathbb{P}_{b}}(1)$, as it can be bounded uniformly over $b,h$ using
(\ref{eqinvariantmeasure}):
\begin{equation*}
\everymath{\displaystyle}
\begin{array}[b]{l}
\lvert \log \mu_{b+h/\sqrt{n}}(X_{0})-\log \mu_{b}(X_{0})\rvert
\lesssim \|\mu_{b+h/\sqrt{n}}-\mu_{b}\|_{\infty }
\\\noalign{\vspace{4pt}}
\;\;
\lesssim \Big \|\frac{e^{\int_{0}^{\cdot }\left( b+h/\sqrt{n}\right) (y)dy}}{
\int_{0}^{1}e^{\int_{0}^{x}\left( b+h/\sqrt{n}\right) (y)dy}dx}-\frac{e
^{\int_{0}^{\cdot }b(y)dy}}{\int_{0}^{1}e^{\int_{0}^{x}b(y)dy}dx}
\Big \|_{\infty }\xrightarrow{n\to \infty }0.
\end{array}   \qedhere
\end{equation*}
\end{proof}

\section{Local approximation of transition densities}\label{seclocalapprox}
In this section, we study the differentiability properties of
$p_{t,b}(x,y)$ as a function of the drift $b$, and the main goal is to
prove Theorem \ref{thmderivative}. For technical reasons, we first prove
a regularized version of it (Lemma \ref{lem-reg-main} in Section
\ref{subsecreg}) and then let the regularization parameter $\delta >0$
tend to $0$ to obtain Theorem \ref{thmderivative} (Section
\ref{subsecthm1pf}).

\subsection{Preliminaries and notation}

We begin by introducing some notation and important classical results.

\subsubsection{Some function spaces}

For any integer $s\geq 0$, we equip the Sobolev space
$H^{s}=H^{s}((0,1))$ with the inner product
\begin{equation}
\label{eq-Hs-prod}
\langle g_{1},g_{2}\rangle_{H^{s}}=\langle g_{1},g_{2}\rangle_{L^{2}}+
\langle g_{1}^{(s)},g_{2}^{(s)}\rangle_{L^{2}},
\end{equation}
where $L^{2}$ is the usual space of square integrable functions with
respect to Lebesgue measure. Occasionally it will be convenient to
replace the $L^{2}$--inner product above by the $L^{2}(\mu )$--inner
product, where $\mu $ is the invariant measure of $(X_{t}:t\geq 0)$,
which by (\ref{eqinvmeasbdd}) induces a norm which is equivalent to the
norm induced by (\ref{eq-Hs-prod}).

We will also use the fractional Sobolev spaces $H^{s}$ for real
$s\geq 0$, which are obtained by interpolation, see
\cite{lionsmagenes}. For $s>\frac{1}{2}$, the Sobolev embedding
(\ref{eqsobolevemb}) implies that any function $f\in H^{s}$ extends
uniquely to a continuous function on $[0,1]$. The following standard
interpolation equalities and embeddings (see
\cite{lionsmagenes}, p.44-45) will be used throughout. For all
$s_{1},s_{2}\geq 0$ and $\theta \in (0,1)$, we have
\begin{equation}
\label{eqsobolevinter}
\forall f\in H^{s_{1}}\cap H^{s_{2}}:
~~
\|f\|_{H^{\theta s_{1}+(1-\theta )s_{2}}}\leq C(\theta ,s_{1},s_{2})
\|f\|_{H^{s_{1}}}^{\theta }\|f\|_{H^{s_{2}}}^{1-\theta },
\end{equation}
and for each $s>1/2$, we have the multiplicative inequality
\begin{equation}
\label{h-mult}
\forall f,g\in H^{s}:
~~
\|fg\|_{H^{s}}\lesssim C(s)\|f\|_{H^{s}} \|g\|_{H^{s}}
\end{equation}
as well as the continuous embedding
\begin{equation}
\label{eqsobolevemb}
H^{s}\subseteq C([0,1]), \quad \|f\|_{\infty }\leq C(s) \|f\|_{H^{s}},
\end{equation}
where $C([0,1])$ denotes the space of continuous functions on
$[0,1]$. Moreover, for any $s>0$, we define the negative order Sobolev
space $H^{-s}$ as the topological dual space of $H^{s}$, where for any
$f\in L^{2}$, the norm can be written as
\begin{equation*}
\|f\|_{H^{-s}}=\sup_{\psi \in H^{s}, \|\psi \|_{H^{s}}\le 1}\big|\int
_{0}^{1} f\psi \big|.
\end{equation*}

For any $T>0$, any Banach space $(X,\|\cdot \|)$ and any integer
$k\geq 0$, we denote by $C^{k}([0,T],X)$ the $k$-times continuously
differentiable functions from $[0,T]$ to $X$, equipped with the norm
\begin{equation*}
\|f\|_{C^{k}([0,T],X)}=\sum_{i=0}^{k}\sup_{t\in [0,T]}\big \|\frac{d
^{i}}{dt^{i}}f(t)\big \|.
\end{equation*}
For $\alpha >0$ with $\alpha \not \in \mathbb{N}$, we denote the space
of $\alpha $-H\"{o}lder continuous functions $f:[0,T]\to X$ by
$C^{\alpha }([0,T],X)$ and equip it with the usual norm
\begin{equation*}
\|f\|_{C^{\alpha }([0,T],X)}=\|f\|_{C^{\lfloor \alpha \rfloor }([0,T],X)}+
\sup_{s,t\in [0,T], s\neq t}\frac{\big \|\frac{d^{\lfloor \alpha
\rfloor }}{dt^{\lfloor \alpha \rfloor }}f(t)-\frac{d^{\lfloor \alpha
\rfloor }}{dt^{\lfloor \alpha \rfloor }}f(s)\big \|}{\big|t-s\big|^{
\alpha -\lfloor \alpha \rfloor }}.
\end{equation*}
We will frequently, without further comment, interpret functions
$f:[0,T]\times [0,1]\to \mathbb{R}$ as $L^{2}$-valued maps $f:[0,T]
\to L^{2}, f(t)=f(t,\cdot )$, and vice versa.

\subsubsection{The differential operator $\mathcal{L}_{b}$}

For any drift function $b\in C^{1}_0$, we define the differential operator
\begin{equation*}
\begin{split}
\mathcal{L}_{b}f(x):=
&\; f''(x)+b(x)f'(x)
\qquad
\text{for }f\in \mathcal{D},
\\
\mathcal{D}:=
&\left\{ f\in H^{2}: f'(0)=f'(1)=0\right\} .
\end{split}
\end{equation*}
It is well-known that $\mathcal{L}_{b}$ is the infinitesimal generator
of the semigroup $(P_{t,b}:t\geq 0)$ defined in (\ref{eqsemigroup}), so
that we get by the usual functional calculus that $P_{t,b}=e^{t
\mathcal{L}_{b}}$ for all $t\geq 0$ (with the convention $e^{0}=Id$).
The fact that the domain $\mathcal{D}$ of $\mathcal{L}_{b}$ is equipped
with Neumann boundary conditions corresponds to the diffusion being
reflected at the boundary, see \cite{hansenetal} for a detailed
discussion. We equip $\mathcal{D}$ with the graph norm
\begin{equation*}
\|f\|_{\mathcal{L}_{b}}:=\big(\|\mathcal{L}_{b}f\|^{2}_{L^{2}(\mu_{b})}+
\| f\|^{2}_{L^{2}(\mu_{b})}\big)^{1/2},
\end{equation*}
which by Lemma \ref{lemnormequiv} is equivalent to the $H^{2}$-norm on
$\mathcal{D}$. Moreover, for $h\in H^{1}$, we define the first order
differential operator
\begin{equation}
\label{eqLh}
L_{h}f(x)=h(x)f'(x),
\qquad
f\in H^{1}.
\end{equation}

The operator $\mathcal{L}_{b}$ has a purely discrete spectrum
$\text{Spec}(\mathcal{L}_{b})\subseteq (-\infty ,0]$ (see
\cite{daviesspectral}, Theorem 7.2.2). We will denote by $(u_{j,b})_{j
\geq 0}$ the $L^{2}(\mu_{b})$-normalized orthogonal basis of
$L^{2}(\mu_{b})$ consisting of the eigenfunctions $u_{j,b}\in
\mathcal{D}$ of $\mathcal{L}_{b}$, ordered such that the corresponding
eigenvalues $(\lambda_{j,b})_{j\geq 0}$ are non-increasing. When there
is no ambiguity, we will often simply write $\lambda_{j}$ and
$u_{j}$. We will use throughout the spectral decomposition
\begin{equation}
\label{eqspectral}
p_{t,b}(x,y)=\sum_{j\geq 0}e^{\lambda_{j} t}u_{j}(x)u_{j}(y)\mu (y),
\quad x,y\in [0,1],~ t>0,
\end{equation}
see e.g. p. 101 in \cite{bgl}, and the spectral representations
\begin{align}
&\mathcal{L}_{b}f=\sum_{j\ge 0}\lambda_{j}\langle f, u_{j}
\rangle_{L^{2}(\mu )}u_{j},
~~
f\in \mathcal{D},
\label{Lb-spectral}\\
&P_{t,b}f=\sum_{j\ge 0}e^{t\lambda_{j}}\langle f, u_{j}
\rangle_{L^{2}(\mu )}u_{j},
~~
f\in L^{2}, ~t>0.
\label{Pt-spectral}
\end{align}
We also note that (\ref{eqinvariantmeasure}) immediately yields that
there exist constants $0<C<C'<\infty $ such that for all $b\in \Theta
$ and all $x\in [0,1]$,
\begin{equation}
\label{eqinvmeasbdd}
C\leq \mu_{b}(x)\leq C'.
\end{equation}

\subsubsection{A key PDE result}

For any $f\in C([0,T],L^{2})$ and $u_{0}\in \mathcal{D}$, consider the
inhomogeneous parabolic equation
\begin{equation}
\label{eqparabolic}
\begin{cases}
\frac{d}{dt}u(t)=\mathcal{L}_{b}u(t)+f(t)
\qquad
\text{for all }t\in [0,T],
\\
u(0)=u_{0}.
\end{cases}
\end{equation}
We say that a function $u: [0,T]\to L^{2}$ is a solution to
(\ref{eqparabolic}) if $ u\in C^{1}([0,T],L^{2})\cap C([0,T],
\mathcal{D})$ and (\ref{eqparabolic}) holds. The next proposition
regarding the existence, uniqueness and regularity properties of
solutions to (\ref{eqparabolic}) will play a key role for the proofs in
the rest of Section \ref{seclocalapprox}. To state the result, we need
the following interpolation spaces $\mathcal{D}(\alpha ), 0\leq
\alpha \leq 1$, between $L^{2}$ and $\mathcal{D}$:
\begin{align}
&\mathcal{D}(\alpha ):=\big \{f\in L^{2}:\omega (t):=t^{-\alpha }
\left\| P_{t,b}f-f\right\| _{L^{2}(\mu_{b})} \text{ is bounded on }t
\in (0,1] \big \},
\nonumber
\\
&\|f\|_{\mathcal{D}(\alpha )}:=\|f\|_{L^{2}(\mu_{b})}+\sup_{t\in [0,1]}
\omega (t).
\label{eqDalpha}
\end{align}
\begin{prop}
\label{proplunardi}
Suppose $0<\alpha <1$, $f\in C^{\alpha }([0,T],L^{2})$ and $u_{0}
\in \mathcal{D}$. Then there exists a unique solution $u$ to
(\ref{eqparabolic}), given by the Bochner integral
\begin{equation}
\label{eqduhamel}
u(t)=P_{t,b}u_{0}+\int_{0}^{t}P_{t-s,b}f(s)ds,
~~
t\in [0,T].
\end{equation}
If also $f(0)+\mathcal{L}_{b}u_{0}\in \mathcal{D}(\alpha )$, then we
have $u\in C^{1+\alpha }([0,T],L^{2})\cap C^{\alpha }([0,T],
\mathcal{D})$ and there exists $C<\infty $ so that for all such $f$ and
$u_{0}$,
\begin{equation*}
\begin{split}
\|u\|_{C^{1+\alpha }([0,T],L^{2})}
&+\|u\|_{C^{\alpha }([0,T],
\mathcal{D})}
\\
&\leq C\left( \|f\|_{C^{\alpha }([0,T],L^{2})}+\|u_{0}\|_{\mathcal{L}
_{b}}+\|f(0)+\mathcal{L}_{b}u_{0}\|_{\mathcal{D}(\alpha )}\right) .
\end{split}
\end{equation*}
\end{prop}
\begin{proof}
This is a special case of Theorem 4.3.1 (iii) in \cite{lunardi}
with $X=L^{2}(\mu_{b})$ and $A=\mathcal{L}_{b}$, where we note that the
integral formula (\ref{eqduhamel}) is given by Proposition 4.1.2 in the
same reference. We also note that $\mathcal{D}(\alpha )$ coincides with
the space $D_{A}(\alpha ,\infty )$ from \cite{lunardi} with
equivalent norms, see Proposition 2.2.4 in \cite{lunardi}. It
therefore suffices to verify that the general theory for parabolic PDEs
developed in \cite{lunardi} applies to our particular case. For
that, we need to check that $(P_{t,b}:t\geq 0)$ is an analytic semigroup
of operators on $L^{2}$ in the sense of \cite{lunardi}, p.34,
which requires the following.
\begin{enumerate}
\item
For some $\theta \in (\pi /2,\pi )$ and $\omega \in \mathbb{R}$, the
resolvent set $\rho (\mathcal{L}_{b})$ of $\mathcal{L}_{b}$ contains the
sector $S_{\theta , \omega }\subseteq \mathbb{C}$, where $S_{\theta ,
\omega }$ is defined by
\begin{equation*}
\label{eqsectorial1}
S_{\theta ,\omega }:=\{\lambda \in \mathbb{C}: \lambda \neq \omega , |
\arg (\lambda -\omega )|<\theta \}.
\end{equation*}

\item
There exists some $M<\infty $ such that we have the resolvent estimate
\begin{equation*}
\|R(\lambda ,\mathcal{L}_{b})\|_{L^{2}\to L^{2}}\leq M|\lambda -
\omega |^{-1}
~~~
\forall \lambda \in S_{\theta ,\omega }.
\end{equation*}
\end{enumerate}
As $\mathcal{L}_{b}$ has a discrete spectrum contained in the
non-positive half line, both of the above properties are easily checked
with $\omega =0$ and any $\theta \in (\frac{\pi }{2},\pi )$.
\end{proof}

\begin{defin}[Solution operator]
In what follows, we denote by $\mathcal{S}= ( \frac{d}{dt}-\break
\mathcal{L}_{b} ) ^{-1}$ the linear solution operator which maps any
$f\in C^{\alpha }([0,T],L^{2})$, $0<\alpha <1$, to the unique solution
$u= \mathcal{S}(f)$ of the parabolic problem
\begin{equation}
\label{Sdef}
\begin{cases}
\left( \frac{d}{dt}-\mathcal{L}_{b}\right) u(t)=f(t),\quad t\in [0,T],
\\
u(0)=0.
\end{cases}
\end{equation}
\end{defin}

\subsection{Approximation of regularized transition densities}\label{subsecreg}
The main result of this section is Lemma \ref{lem-reg-main}, which can
be viewed as a regularized version of Theorem \ref{thmderivative}. The
main tools used to prove it are Proposition \ref{proplunardi} as well
as the spectral analysis of $\mathcal{L}_{b}$ and $P_{t,b}$ from Section
\ref{secspectral}.

In order to apply Proposition \ref{proplunardi}, we view the transition
densities $p_{t,b}(x,y)$ as functions of the two variables $(t,x)
\in [0,T]\times [0,1]$ with $y\in [0,1]$ fixed, where $T$ is an
arbitrary constant $T>\Delta >0$ (with the convention that $p_{0,b}(
\cdot ,y)$ is the point mass at $y$). Due to the singular behaviour of
$p_{t,b}(x,y)$ for $(t,x)\to (0,y)$, a regularisation argument is
needed. For any $\delta >0$ and $d\in C^{1}_{0}$, define the
$\delta $- regularized transition densities by
\begin{equation*}
\begin{split}
u_{d}^{\delta }: [0,\infty )\times [0,1]\to \mathbb{R},
~~
u_{d}^{\delta }(t,x):=P_{\delta ,0}(p_{t,d}(x,\cdot ))(y),
\end{split}
\end{equation*}
where $(P_{t,0}:t\geq 0)$ denotes the transition semigroup for $b=0$,
which corresponds to reflected Brownian motion.

\subsubsection{Recursive definition of approximations}\label{sec-rec-def}
We now implicitly define the `candidate' local approximations to
$u_{d}^{\delta }$ as solutions to certain parabolic PDEs. To that end,
we note that using (\ref{eqsemigroup}), one easily checks that for all
$t\geq 0$,
\begin{equation}
\label{eq-phi-delta}
u_{d}^{\delta }(t)=P_{t,d}\varphi_{\delta },
~~
\text{where}
~~
\varphi_{\delta }(x):=p_{\delta ,0}(y,x).
\end{equation}

Hence we can give the following crucial PDE interpretation to
$u_{d}^{\delta }$.
\begin{lem}
\label{lemudelta}
For any $d\in C^{1}_{0}$, we have that $u_{d}^{\delta }\in C^{3/2}([0,T],L
^{2})\cap C^{1/2}([0,T],\mathcal{D})$, and $u_{d}^{\delta }$ is the
unique solution to the initial value problem
\begin{equation}
\label{equdeltapde}
\begin{cases}
\mathcal{(}\frac{d}{dt}-\mathcal{L}_{d})u(t)=0 \;\text{ for all }t
\in [0,T],
\\
u(0)=\varphi_{\delta }.
\end{cases}
\end{equation}
\end{lem}
\begin{proof}
We check that Proposition \ref{proplunardi} applies with $\alpha =1/2$.
For this, we need that $\varphi_{\delta }\in \mathcal{D}$ and that
$\mathcal{L}_{d}\varphi_{\delta }\in \mathcal{D}(1/2)$. Using the
spectral decomposition (\ref{eqspectral}) and the fact that
$\mu_{b}=\text{Leb}([0,1])$ for $b=0$, we see by differentiating under
the sum that $\varphi_{\delta }\in \mathcal{D}$. This is possible by
Lemma \ref{lem bounds on eigenfunctions} and the dominated convergence
theorem. The same argument yields that $\varphi_{\delta }\in H^{3}$.
Thus, we have that $\mathcal{L}_{d}\varphi_{\delta }\in H^{1}$, which
is a subset of $\mathcal{D}(1/2)$ by the second part of Lemma
\ref{lemPtsmoothing}.
\end{proof}

We now recursively define the functions $R_{k}^{\delta }[h]$ and
$v_{k}^{\delta }[h]$, $k\geq 0$. The norm estimates in Section
\ref{sec-reg-est} justify that they are the correct remainder and
approximating terms, respectively, in the $k$-th order Taylor expansion
of $\eta \mapsto u_{b+\eta h}^{\delta }$.

\begin{defin}
Let $b,h\in C^{1}_{0}$ and $\delta >0$.

1. For $k=0$, we define the `$0$-th order local approximation' of
$\eta \mapsto u^{\delta }_{b+\eta h}$ at $0$, and the remainder of this
approximation, by
\begin{equation*}
v_{0}^{\delta }[h]=v_{0}^{\delta }:=u_{b}^{\delta },~~~~ R_{0}^{
\delta }[h]=R_{0}^{\delta }:=u_{b+h}^{\delta }-u_{b}^{\delta }.
\end{equation*}

2. For $k\geq 1$, we recursively define the functions $R_{k}^{\delta
}[h]=R_{k}^{\delta }, v_{k}^{\delta }[h]=v_{k}^{\delta }\in C^{3/2}([0,T],L
^{2})\cap C^{1/2}([0,T],\mathcal{D})$ by
\begin{equation}
\label{rk-vk-def}
R^{\delta }_{k}[h]:= \mathcal{S}\big(L_{h}R_{k-1}^{\delta }[h]\big),
~~~
v_{k}^{\delta }[h]:=R^{\delta }_{k-1}[h]-R_{k}^{\delta }[h],
\end{equation}
where $\mathcal{S}$ is the solution operator defined in (\ref{Sdef}) and
$L_{h}$ was defined in (\ref{eqLh}).
\end{defin}
We should justify why the definition (\ref{rk-vk-def}) is admissible,
and we do so by induction. By Lemma \ref{lemudelta}, we have
$R_{0}^{\delta }[h]\in C^{3/2}([0,T],L^{2})\cap C^{1/2}([0,T],
\mathcal{D})$. Hence, using the definition of $R_{k}^{\delta }[h]$ and
Proposition \ref{proplunardi} inductively, we obtain that for all
$k\ge 1$, $L_{h}R^{\delta }_{k-1}[h]\in C^{1/2}([0,T],H^{1})$ as well
as $L_{h}R^{\delta }_{k-1}[h](0)=0$, so that $R_{k}^{\delta }$,
$v_{k}^{\delta }$ have the stated regularity. Thus, (\ref{rk-vk-def})
is well-defined.

By definition of $\mathcal{L}_{b}$ and (\ref{equdeltapde}), we see that
$R_{0}^{\delta }[h]$ is the unique solution to
\begin{equation}
\label{eqR0PDE}
\mathcal{(}\frac{d}{dt}-\mathcal{L}_{b})R_{0}^{\delta }(t)= L_{h}u
^{\delta }_{b+h}(t)
~~
\forall t\in [0,T]
~~~
\textnormal{and}
~~~
R_{0}^{\delta }(0)=0,
\end{equation}
and (\ref{rk-vk-def}) yields that
\begin{equation}
\label{poly-appr}
u^{\delta }_{b+h}=\sum_{i=0}^{k}v_{i}^{\delta }[h]+R_{k}^{\delta }[h]
~~~~
\forall b,h\in C^{1},
~~
k\ge 0.
\end{equation}
The regularity estimates for $R_{k}^{\delta }[h]$ in the next section
will justify that (\ref{poly-appr}) is in fact the Taylor approximation
for $\eta \mapsto u^{\delta }_{b+\eta h}$. Before proceeding to this,
we need to check that the $v_{k}^{\delta }[h]$ are homogeneous of degree
$k$ in $h$, i.e. that
\begin{equation}
\label{eqhomogenous}
\forall h\in C^{1}_{0}~\forall \eta \in \mathbb{R}:
~~
v_{k}^{\delta }[\eta h]=\eta^{k}v_{k}^{\delta }[h].
\end{equation}
This is again seen by induction. For $k=0$, we have that $v_{0}^{
\delta }[\eta h]=u^{\delta }_{b}=v_{0}^{\delta }[h]$, and if
(\ref{eqhomogenous}) holds for some $k\geq 0$, then we have that
\begin{equation*}
v_{k+1}^{\delta }[\eta h]=\mathcal{S}(L_{\eta h}v_{k}^{\delta }[
\eta h])=\eta^{k+1} \mathcal{S}(L_{h}v_{k}^{\delta }[h]),
\end{equation*}
where we have used that for each $k\in \mathbb{N}\cup \{0\}$,
\begin{equation*}
\big(\frac{d}{dt}-\mathcal{L}_{b}\big)v_{k+1}^{\delta }=L_{h}(R_{k-1}
^{\delta }-R_{k}^{\delta })=L_{h}v_{k}^{\delta }.
\end{equation*}

\subsubsection{Regularity estimates}\label{sec-reg-est}
We now derive norm estimates for the remainders $R_{k}^{\delta }[h]$
from (\ref{rk-vk-def}) and (\ref{poly-appr}), using Propsition
\ref{proplunardi} and the results from Section \ref{secspectral}.

The following Lemma is the main result of Section
\ref{seclocalapprox}. It can be viewed as a regularised version of
Theorem \ref{thmderivative}. Crucially, the estimate below is uniform
in $\delta >0$ such that it can be preserved in the limit $\delta
\to 0$.
\begin{lem}
\label{lem-reg-main}
For each $\epsilon >0$, there exists $C>0$ such that for all
$b\in \Theta $ from (\ref{theta-def}), $h\in C^{1}_{0}$ with
$\|h\|_{H^{1}}\leq 1$, $y\in [0,1]$, $k\in \mathbb{N}\cup \{0\}$ and
$\delta >0$,
\begin{equation*}
\|R_{k}^{\delta }[h](\Delta )\|_{\infty }\leq C^{k}\|h\|_{H^{1}}^{k+1/2-
\epsilon }.
\end{equation*}
\end{lem}
The rest of this section is concerned with proving Lemma
\ref{lem-reg-main}. In what follows, when we write that an inequality
is `uniform' without further comment, or when we use the symbols
$\lesssim , \gtrsim , \simeq $, we mean that the constants involved can
be chosen uniformly over $b,h,y,k$ and $\delta $ as in the statement of
Lemma \ref{lem-reg-main}.

The proof of Lemma \ref{lem-reg-main} consists of two separate lemmas,
which establish an $L^{2}$-estimate (\ref{eqRkreg}) and an
$H^{1}$-estimate (\ref{rk-h1-reg}) for $R_{k}^{\delta }[h](\Delta )$
respectively. Given these two estimates, Lemma \ref{lem-reg-main} then
immediately follows from interpolating -- Indeed, taking $C$ to be the larger
of the two constants from (\ref{eqRkreg}) and (\ref{rk-h1-reg}) yields
\begin{equation*}
\|R_{k}[h](\Delta )\|_{\infty }\lesssim \|R_{k}[h](\Delta )\|_{H^{
\frac{1}{2}+\varepsilon }}\lesssim \|R_{k}(\Delta )\|_{L^{2}}^{
\frac{1}{2}-\varepsilon }\|R_{k}(\Delta )\|_{H^{1}}^{\frac{1}{2}+
\varepsilon }\leq C^{k}\|h\|_{H^{1}}^{k+\frac{1}{2}-\epsilon }.
\end{equation*}

\paragraph{The $L^{2}$-estimate}

To obtain estimates which are uniform in $\delta >0$, we `regularise'
$R_{k}^{\delta }$ further by integrating in time. For $k\geq 0$, define
\begin{equation*}
Q_{k}^{\delta }[h]:[0,T]\to L^{2},\quad Q_{k}^{\delta }[h](t):=\int
_{0}^{t}R_{k}^{\delta }[h](s)ds.
\end{equation*}
Here is the $L^{2}$-estimate.
\begin{lem}
\label{lemregest}
1. Let $b,h\in C^{1}_{0}$, $\delta >0$ and recall the definition
(\ref{Sdef}) of $\mathcal{S}$. Then we have that
\begin{equation}
\label{eqQ0PDE}
Q_{0}^{\delta }[h]= \mathcal{S}\big(L_{h}\int_{0}^{\cdot }u_{b+h}^{
\delta }(s)ds\big),
\end{equation}
and for $k\geq 1$, we have that
\begin{equation}
\label{eqQkPDE}
Q_{k}^{\delta }[h]=\mathcal{S}\big(L_{h}Q_{k-1}^{\delta }[h]\big).
\end{equation}

2. For all $\alpha <1/4$, there exists $C< \infty $ such that for all
$b,h,y,k,\delta $ as in Lemma \ref{lem-reg-main},
\begin{equation}
\label{eqQkreg}
\|Q_{k}^{\delta }[h]\|_{C^{1+\alpha }([0,T],L^{2})}+\|Q_{k}^{\delta }[h]
\|_{C^{\alpha }([0,T],\mathcal{D})}\leq C^{k}\|h\|_{\infty }^{k+1}.
\end{equation}
In particular, we have that
\begin{equation}
\label{eqRkreg}
\|R_{k}^{\delta }[h]\|_{C^{\alpha }([0,T],L^{2})}\leq C^{k}\|h\|_{
\infty }^{k+1}.
\end{equation}

\end{lem}
\begin{proof}
We first show (\ref{eqQ0PDE}). Using Riemann sums to approximate the
integrals below, the closedness of the operators $\mathcal{L}_{b}$ and
$L_{h}$ as well as (\ref{eqR0PDE}), we obtain that
\begin{equation}
\label{eqQ0PDE2}
\begin{split}
\big(\frac{d}{dt}-\mathcal{L}_{b}\big)Q_{0}^{\delta }
&=R_{0}^{\delta
}(t)-\int_{0}^{t}\mathcal{L}_{b}R_{0}^{\delta }(s)ds=R_{0}^{\delta }(t)-R
_{0}^{\delta }(0)-\int_{0}^{t}\mathcal{L}_{b}R_{0}^{\delta }(s)ds
\\
&=\int_{0}^{t}\big(\frac{d}{ds}-\mathcal{L}_{b}\big)R_{0}^{\delta }(s)ds=L
_{h}\int_{0}^{t}u^{\delta }_{b+h}(s)ds.
\end{split}
\end{equation}
Moreover, we have $Q_{0}^{\delta }(0)=0$ and $Q_{0}^{\delta }\in C
^{3/2}([0,T],L^{2})\cap C^{1/2}([0,T],\mathcal{D})$, so that
(\ref{eqQ0PDE}) follows from Proposition \ref{proplunardi}. For
$k\ge 1$, (\ref{eqQkPDE}) is proved in the same manner.

Next, we prove (\ref{eqQkreg}) for $k=0$. Let $\alpha <1/4$,
$\delta >0$, $b\in \Theta $, $\|h\|\in C^{1}_{0}$ with $\|h\|_{H^{1}}
\leq 1$, and let us write
\begin{equation*}
f(t)=\partial_{x}\Big(\int_{0}^{t}u^{\delta }_{b+h}(s)ds\Big).
\end{equation*}
In view of (\ref{eqQ0PDE}) and Proposition \ref{proplunardi}, and noting
that $hf(0)=0$, it suffices to show that $\|f\|_{C^{\alpha }([0,T],L
^{2})}\leq C$ for some uniform constant $C$. For all $t<t'\in [0,T]$,
we have by the definition of $u_{b+h}^{\delta }$ and Fubini's theorem
that
\begin{equation*}
\begin{split}
\left[ f(t')-f(t)\right] (x)=
&\partial_{x}\int_{t}^{t'}\int_{0}^{1}p
_{s,b+h}(x,z)\varphi_{\delta }(z)dzds
\\
=
&\partial_{x}\int_{0}^{1}\Big(\int_{t}^{t'}p_{s,b+h}(x,z)ds \Big)
\varphi_{\delta }(z)dz.
\end{split}
\end{equation*}
For convenience, let us for now write $\mu $ for $\mu_{b+h}$ and
$(\lambda_{j},u_{j})_{j\geq 0}$ for the eigenpairs of $\mathcal{L}
_{b+h}$. Using the spectral decomposition (\ref{eqspectral}) with
$b+h$ in place of $b$ and Fubini's theorem, integrating each summand
separately yields that
\begin{equation}
\label{eqf(t)difference}
\begin{split}
\left[ f(t')-f(t)\right] (x)
&=(t'-t)\partial_{x}\int_{0}^{1}
\varphi_{\delta }(z)\mu (z)dz
\\
&
\qquad
+\partial_{x}\int_{0}^{1} \sum_{j\geq 1} \frac{1}{\lambda_{j}}(e^{t'
\lambda_{j}}-e^{t\lambda_{j}})u_{j}(x) u_{j}(z)\varphi_{\delta }(z)
\mu (z)dz
\\
&=\partial_{x}\sum_{j\geq 1} \frac{1}{\lambda_{j}}(e^{t'\lambda_{j}}-e
^{t\lambda_{j}})u_{j}(x) \langle u_{j},\varphi_{\delta }
\rangle_{L^{2}(\mu )},
\end{split}
\end{equation}
where Fubini's theorem is applicable due to Lemma
\ref{lem bounds on eigenfunctions}:
\begin{equation*}
\begin{split}
\sum_{j\geq 1} \Big|\frac{1}{\lambda_{j}}(e^{t'\lambda_{j}}-e^{t
\lambda_{j}})u_{j}(x) \langle u_{j},\varphi_{\delta }\rangle_{L^{2}(
\mu )}\Big| \lesssim \|\mu \varphi_{\delta }\|_{L^{2}}\sum_{j\geq 1}j
^{-2}\|u_{j}\|_{\infty }\lesssim \sum_{j\geq 1}j^{-3/2+\varepsilon }.
\end{split}
\end{equation*}
From Lemma \ref{lemneumannregularity} and (\ref{eqf(t)difference}), it
follows that
\begin{equation*}
f(t')-f(t)=\partial_{x} \left(  \mathcal{L}_{b+h}^{-1}\left( P_{t',b+h}-P
_{t,b+h}\right) \varphi_{\delta }\right) .
\end{equation*}
Using this, (\ref{ell-norm-est}), the self-adjointness of $P_{t,b+h}$
with respect to $\langle \cdot ,\cdot \rangle_{L^{2}(\mu )}$ and
(\ref{Pt-H1}), we obtain that
\begin{equation*}
\begin{split}
\|f(t')-f(t)\|_{L^{2}}
&\leq \| \mathcal{L}_{b+h}^{-1}\left( P_{t',b+h}-P
_{t,b+h}\right) \varphi_{\delta }\|_{H^{1}}
\\
&\lesssim \|P_{t,b+h}(P_{t'-t,b+h}-Id)\varphi_{\delta }\|_{H^{-1}}
\\
&\lesssim \sup_{\phi \in H^{1}, \|\phi \|_{H^{1}}\leq 1}\left| \langle
(P_{t'-t,b+h}-Id)P_{t,b+h}\varphi_{\delta },\phi \rangle_{L^{2}(
\mu )}\right|
\\
&=\sup_{\phi \in H^{1}, \|\phi \|_{H^{1}}\leq 1}\left| \langle P_{t,b+h}
\varphi_{\delta },P_{t'-t,b+h}\phi -\phi \rangle_{L^{2}(\mu )}\right|
\\
&\lesssim \sup_{t>0}\|P_{t,b+h}\varphi_{\delta }\|_{L^{1}}
\sup_{\phi \in H^{1}, \|\phi \|_{H^{1}}\leq 1}\left\| P_{t'-t,b+h}
\phi -\phi \right\| _{\infty }
\\
&\lesssim \sup_{\phi \in H^{1}, \|\phi \|_{H^{1}}\leq 1}\|P_{t'-t,b+h}
\phi -\phi \|_{\infty }
\\
&\lesssim (t'-t)^{\alpha }.
\end{split}
\end{equation*}
Hence, Proposition \ref{proplunardi} and (\ref{eqQ0PDE}) imply
(\ref{eqQkreg}) for $k=0$. Choosing $C$ large enough and inductively
using Proposition \ref{proplunardi} and (\ref{eqQkPDE}), we also obtain
(\ref{eqQkreg}) for $k\geq 1$:
\begin{equation*}
\begin{split}
\|Q_{k}^{\delta }\|_{C^{\alpha }([0,T],\mathcal{D})}
&+\|Q_{k}^{\delta
}\|_{C^{1+\alpha }([0,T],L^{2})}\le C\|L_{h}Q_{k-1}^{\delta }\|_{C
^{\alpha }([0,T],L^{2})}
\\
&\le C\|h\|_{\infty }\|Q_{k-1}^{\delta }\|_{C^{\alpha }([0,T],
\mathcal{D})}\leq C^{k} \|h\|_{\infty }^{k+1}.
\end{split}
\end{equation*}
Finally, (\ref{eqRkreg}) follows upon differentiating (\ref{eqQkreg})
in $t$.
\end{proof}

\paragraph{The $H^{1}$ estimate}

The $H^{1}$-estimate reads as follows.
\begin{lem}
\label{lemH1estimate}
Let $k\geq 0$ be an integer and $\Delta >0$. Then there exists
$C<\infty $ such that for all $b,h,y,k,\delta $ as in Lemma
\ref{lem-reg-main},
\begin{equation}
\label{rk-h1-reg}
\|R_{k}^{\delta }(\Delta )\|_{H^{1}}\leq C^{k}\|h\|_{H^{1}}^{k}.
\end{equation}
\end{lem}

To prove Lemma \ref{lemH1estimate}, we express $R_{k}^{\delta }[h]$
using (\ref{eqduhamel}) and decompose the integral into times close to
$0$ and times bounded away from $0$. The following Lemma allows us to
control the respective integrals.

\begin{lem}
For any $T>0$, there exists $C<\infty $ such that for
all $b,h,y,k,\delta $ as in Lemma \ref{lem-reg-main} and all $\tilde{T}\in (0,T)$, we have the estimates
\begin{align}
\label{eqH1term1}
\Big \| \int_{0}^{\tilde{T}}P_{T-s}L_{h}R_{k}^{\delta }(s)ds \Big \|
_{H^{1}} &\leq \frac{C}{(T-\tilde{T})^{5/4}}\|h\|_{H^{1}}
\sup_{s\in [0,\tilde{T}]}\|R_{k}^{\delta }(s)\|_{L^{2}},\\
\label{eqH1term2}
\Big \| \int_{\tilde{T}}^{T}P_{T-s}L_{h}R_{k}^{\delta }(s)ds \Big \|
_{H^{1}} &\leq C\|h\|_{\infty }\sup_{s\in [\tilde{T},T]}\|R_{k}^{
\delta }(s)\|_{H^{1}}.
\end{align}
\end{lem}

\begin{proof}
We first show (\ref{eqH1term2}). By Lemma \ref{lemnormequiv}, we can
estimate the $(-\mathcal{L}_{b})^{1/2}$-graph norm instead of the
$H^{1}$ norm. Using Lemma \ref{lem bounds on eigenfunctions}, we have
\begin{equation*}
\begin{split}
&\Big \| (-\mathcal{L}_{b})^{1/2}\int_{\tilde{T}}^{T}P_{T-s}L_{h}R
_{k}^{\delta }(s)ds \Big \|_{L^{2}(\mu )}^{2}
\\
&
\qquad
=\sum_{j=1}^{\infty }\Big(\int_{\tilde{T}}^{T} |\lambda_{j}|^{
\frac{1}{2}}e^{\lambda_{j}(T-s)}\langle u_{j},hR_{k}^{\delta }(s)'
\rangle_{L^{2}(\mu )}ds\Big)^{2}
\\
&
\qquad
\lesssim \sum_{j=1}^{\infty }\Big(\int_{\tilde{T}}^{T}je^{-cj^{2}(T-s)}
\|hR_{k}^{\delta }(s)'\|_{L^{2}}ds\Big)^{2}
\\
&
\qquad
\lesssim \|h\|_{\infty }^{2}\sup_{s\in [\tilde{T},T]}\|R_{k}^{\delta
}(s)\|_{H^{1}}^{2} \sum_{j=1}^{\infty }\Big(j\int_{\tilde{T}}^{T}e
^{-cj^{2}(T-s)}ds\Big)^{2}
\\
&
\qquad
\lesssim \|h\|^{2}_{\infty }\sup_{s\in [\tilde{T},T]}\|R_{k}^{\delta
}(s)\|_{H^{1}}^{2}\sum_{j=1}^{\infty }\frac{1}{j^{2}}.
\end{split}
\end{equation*}
A similar calculation yields that
\begin{equation*}
\begin{split}
\Big \|\int_{\tilde{T}}^{T}P_{T-s}L_{h}R_{k}^{\delta }(s)ds \Big \|
_{L^{2}(\mu )}^{2}\lesssim \|h\|^{2}_{\infty }\sup_{s\in [\tilde{T},T]}
\|R_{k}^{\delta }(s)\|_{H^{1}}^{2} \Big(T^{2}+\sum_{j=1}^{\infty }\frac{1}{j
^{4}}\Big).
\end{split}
\end{equation*}
Combining the last two displays completes the proof of
(\ref{eqH1term2}).

Next, we prove (\ref{eqH1term1}). Using (\ref{Pt-H1H-1}) with
$\alpha =1$, the boundary condition $h(0)=h(1)=0$ to integrate by parts
and (\ref{eqsobolevinter}), we obtain
\begin{equation*}
\everymath{\displaystyle}
\begin{array}[b]{r@{\;}l}
\Big \|\int_{0}^{\tilde{T}}
&P_{T-s}L_{h}R_{k}^{\delta }(s)ds\Big \|
_{H^{1}}\lesssim \int_{0}^{\tilde{T}}(T-s)^{-\frac{5}{4}}\left\| L_{h}R
_{k}^{\delta }(s)\right\| _{H^{-1}}ds
\\\noalign{\vspace{4pt}}
&\le (T-\tilde{T})^{-\frac{5}{4}}\int_{0}^{\tilde{T}}
\sup_{\psi \in C^{\infty },\|\psi \|_{H^{1}}\leq 1}\big| \int_{0}^{1}
(\psi h)' R_{k}^{\delta }(s)\big|ds
\\\noalign{\vspace{4pt}}
&\lesssim (T-\tilde{T})^{-\frac{5}{4}}\|h\|_{H^{1}}
\sup_{s\in [0, \tilde{T}]}\left\| R_{k}^{\delta }(s)\right\| _{L^{2}}.
\end{array}             \qedhere
\end{equation*}
\end{proof}

\begin{proof}[Proof of Lemma \ref{lemH1estimate}]
The case $k=0$ follows from Lemma
\ref{lem bounds on transition densities}. For $k\geq 1$, we iteratively
apply the estimates (\ref{eqH1term1}) and (\ref{eqH1term2}). We first
define the points $\Delta_{j}$ at which we will split the integrals
involved below:
\begin{equation*}
\Delta_{j}:= \Delta \frac{1+j/k}{2},~~ j=0,...,k, ~~~ \text{and}~~~
\eta_{k}:=\frac{\Delta }{2k}=\Delta_{k}-\Delta_{k-1}.
\end{equation*}
Then, using (\ref{rk-vk-def}) and (\ref{eqduhamel}), we can estimate
\begin{equation*}
\begin{split}
\|R_{k}^{\delta }(\Delta )\|_{H^{1}}
&\leq \Big \|\int_{0}^{\Delta_{k-1}}P
_{\Delta -s}L_{h}R_{k-1}^{\delta }(s)ds \Big \|_{H^{1}}+ \Big \|
\int_{\Delta_{k-1}}^{\Delta }P_{\Delta -s}L_{h}R_{k-1}^{\delta }(s)ds
\Big \|_{H^{1}}
\\
&=: I+II.
\end{split}
\end{equation*}
Now let $C$ be the largest of the constants from (\ref{eqRkreg}),
(\ref{eqH1term1}) and (\ref{eqH1term2}). From (\ref{eqH1term1}) with
$\tilde{T}=\Delta_{k-1}$ and (\ref{eqRkreg}), we obtain
\begin{equation*}
I \leq C\eta_{k}^{-\frac{5}{4}} \|h\|_{H^{1}}\sup_{s\in [0,\Delta_{k-1}]}
\left\| R^{\delta }_{k-1}(s)\right\| _{L^{2}}\leq C^{k}\eta_{k}^{-
\frac{5}{4}}\|h\|_{H^{1}}^{k+1}.
\end{equation*}
For the second term, we apply (\ref{eqH1term2}) to obtain
\begin{equation*}
II\leq C\|h\|_{\infty }\sup_{s\in [\Delta_{k-1},\Delta ]}\|R_{k-1}
^{\delta }(s)\|_{H^{1}}.
\end{equation*}
To further estimate the right hand side, we can repeat the argument for
any $s\in [\Delta_{k-1},\Delta ]$:
\begin{equation*}
\begin{split}
\|R_{k-1}^{\delta }(s)\|_{H^{1}}
&\leq \Big \|\int_{0}^{\Delta_{k-2}}\! P
_{\Delta -s}L_{h}R_{k-2}^{\delta }(s)ds \Big \|_{H^{1}}\! +\!  \Big \|
\int_{\Delta_{k-2}}^{s}\!  P_{\Delta -u}L_{h}R_{k-2}^{\delta }(u)du
\Big \|_{H^{1}}
\\
&\le C^{k}\eta_{k}^{-\frac{5}{4}}\|h\|_{H^{1}}^{k}+C\|h\|_{\infty }
\sup_{s\in [\Delta_{k-2},\Delta ]}\|R_{k-2}^{\delta }(s)\|_{H^{1}}.
\end{split}
\end{equation*}
By iterating this argument $k$ times, we obtain that for some larger
constant $\tilde{C}$ independent of $k$,
\begin{equation*}
\|R_{k}^{\delta }(\Delta )\|_{H^{1}}\le kC^{k}\big(\frac{2k}{\Delta }
\big)^{\frac{5}{4}}\|h\|_{H^{1}}^{k+1}+C^{k}\|h\|_{\infty }^{k}
\sup_{s\in [\Delta /2,\Delta ]}\|R_{0}^{\delta }(s)\|_{H^{1}}\le
\tilde{C}^{k}\|h\|_{H^{1}}^{k},
\end{equation*}
where we used (\ref{eq boundedness in H^s}) in the last step. This
completes the proof.
\end{proof}

\subsection{Proof of Theorem \ref{thmderivative}}\label{subsecthm1pf}
We now prove Theorem \ref{thmderivative} by letting $\delta >0$ in Lemma
\ref{lem-reg-main} tend to $0$. Let us fix $b\in \Theta $, $h\in C
^{1}_{0}$ with $\|h\|_{H^{1}}\leq 1$ and $x,y\in [0,1]$, and recall the
notation $\Phi (\eta ):=p_{\Delta , b+\eta h}(x,y)$ for $\eta \in
\mathbb{R}$. For notational convenience, for any $\delta >0, \eta
\in \mathbb{R}$ and integer $k\ge 0$, define
\begin{equation*}
\begin{split}
\Phi^{\delta }(\eta ):=u^{\delta }_{b+\eta h}(\Delta ,x),
~~~
a_{k}^{\delta }:=v_{k}^{\delta }[h](\Delta ,x),
~~~
p_{k}^{\delta }(\eta ):=\sum_{i=0}^{k}a_{i}^{\delta }\eta^{i}.
\end{split}
\end{equation*}
Then by Lemma \ref{lem-reg-main} and (\ref{eqhomogenous}), there exists
$C<\infty $ such that for all $\delta >0$, $k\ge 0$ and
$\eta \in [-1,1]$,
\begin{equation}
\label{eqremainder}
\left| \Phi^{\delta }(\eta )-p_{k}^{\delta }(\eta )\right| =\left| R_{k}
^{\delta }[\eta h](\Delta ,x)\right| \leq \|R_{k}^{\delta }[\eta h](
\Delta )\|_{\infty }\leq C^{k}|\eta |^{k+1/4}.
\end{equation}
Hence for all $\delta >0$, on the interval $\eta \in [-\frac{1}{2C},
\frac{1}{2C}]\cap [-1,1]$, $\Phi^{\delta }$ is given by the power series
$\Phi^{\delta }(\eta )=\sum_{i=0}^{\infty }a_{i}^{\delta }\eta^{i}$. We
divide the rest of the proof into three steps. The first two steps imply
an analogous power series for $\Phi $, and the third proves the integral
formula (\ref{eqderivative}).

\textit{1. Convergence of $\Phi^{\delta }(\eta )$}. Note that by the
definition of $u^{\delta }_{b+\eta h}$, we have that
\begin{equation*}
\forall \eta \in \mathbb{R}:~~\big|\Phi^{\delta }(\eta )-\Phi (\eta )
\big|=\big| P_{\delta ,0}\big(p_{\Delta ,b+\eta h}(x,\cdot )\big)(y)-p
_{\Delta ,b+\eta h}(x,y)\big|.
\end{equation*}
Moreover, by (\ref{eq boundedness in H^s}) we have for any $R>0$ that
\begin{equation}
\label{H1bound}
\sup_{x\in [0,1],\lVert d\rVert_{H^{1}}\leq R}\lVert p_{\Delta ,d}(x,
\cdot )\rVert_{H^{1}}<\infty .
\end{equation}
Thus, using (\ref{Pt-H1}), it follows that for any $\alpha <1/4$, there
is $c<\infty $ such that for all $b\in \Theta , h\in C^{1}_{0}$ with
$\|h\|_{H^{1}}\leq 1$ and $|\eta |\leq 1$,
\begin{equation}
\label{pdeltaconv}
\begin{split}
&\left| P_{\delta ,0}\big(p_{\Delta ,b+\eta h}(x,\cdot )\big)(y)-p_{
\Delta ,b+\eta h}(x,y) \right|
\\
&
\qquad
\leq \sup_{x\in [0,1], \lVert d\rVert_{H^{1}}\leq B+1}\lVert P_{
\delta ,0}p_{\Delta ,d}(x,\cdot )-p_{\Delta ,d}(x,\cdot )\rVert_{
\infty }\leq c\delta^{\alpha } \xrightarrow{\delta \to 0}0.
\end{split}
\end{equation}

\textit{2. Convergence of $a^{\delta }_{k}$}. Fix some $\eta \neq 0$ and
some sequence $\delta_{n}>0$ tending to $0$ as $n\to \infty $. Using
(\ref{eqremainder}), it is easily seen inductively that for all
$k\geq 0$, the sequence $(a^{\delta_{n}}_{k}:n\in \mathbb{N})$ is
bounded. Hence, by a diagonal argument there exists a subsequence
$(\delta_{n_{l}}:l\in \mathbb{N})$ and some sequence $a_{k}\in
\mathbb{R}$ such that for all $k$, $a^{\delta_{n_{l}}}_{k}\xrightarrow{l
\to \infty }a_{k}$. Defining the polynomials
\begin{equation}
\label{pk-def}
p_{k}(\eta ):=\sum_{i=0}^{k}a_{i} \eta^{i},
~~~~
\eta \in \mathbb{R},
~~
k=0,1,2,...,
\end{equation}
we see that (\ref{eqremainder}) still holds with $\Phi $ and
$p_{k}$ in place of $\Phi^{\delta }$ and $p_{k}^{\delta }$. Hence,
$\Phi $ is analytic and $\Phi (\eta )=\sum_{k=0}^{\infty }a_{i}\eta
^{i}$ holds for $\eta \in \left[ -\frac{1}{2C},\frac{1}{2C}\right]
\cap [-1,1]$.

\textit{3. Proof of (\ref{eqderivative})}. It remains to show the
integral formula (\ref{eqderivative}) for $\Phi '(0)$. By what precedes,
we know that the constants $a_{0},a_{1}$ from (\ref{pk-def}) satisfy
\begin{equation*}
\forall \eta \in [-1,1]: |\Phi (\eta )-a_{0}-\eta a_{1}|\leq C|\eta |^{5/4},~~~
\Phi (0)=a_{0}, ~~~\Phi '(0)=a_{1}=\lim_{\delta \to 0}a^{\delta }_{1}.
\end{equation*}
Moreover, by definition of $v^{\delta }_{1}[h]$, we have for all
$\delta >0$ that
\begin{equation*}
a_{1}^{\delta }=v_{1}^{\delta }[h](\Delta ,x)=\mathcal{S}(L_{h}u_{b}
^{\delta })(\Delta ,x)=\int_{0}^{\Delta }\left[ P_{\Delta -s,b}L_{h}u
_{b}^{\delta }(s)\right] (x)ds.
\end{equation*}
Therefore, (\ref{eqderivative}) is proven if we can show that the
following expression converges to $0$ as $\delta \to 0$ (recall that
$\varphi_{\delta }$ was defined in (\ref{eq-phi-delta})):
\begin{equation*}
\begin{split}
&\int_{0}^{\Delta }\big[P_{\Delta -s,b}L_{h}P_{s,b}\varphi_{\delta }
\big](x)ds-\int_{0}^{\Delta }\big[P_{\Delta -s,b}L_{h}p_{s,b}(\cdot ,y)
\big](x)ds
\\
&=\int_{0}^{\Delta }\int_{0}^{1} p_{\Delta -s,b}(x,z)h(z)\partial
_{z}\Big(\int_{0}^{1} p_{s,b}(z,u)\varphi_{\delta }(u)du-p_{s,b}(z,y)
\Big)dzds
\\
&=-\int_{0}^{\Delta /2}\int_{0}^{1} \partial_{z}[p_{\Delta -s,b}(x,z)h(z)]
\Big(\int_{0}^{1} p_{s,b}(z,u)\varphi_{\delta }(u)du-p_{s,b}(z,y)
\Big)dzds
\\
&\quad +\int_{\Delta /2}^{\Delta }\int_{0}^{1} p_{\Delta -s,b}(x,z)h(z)
\partial_{z}\Big(\int_{0}^{1} p_{s,b}(z,u)\varphi_{\delta }(u)du-p
_{s,b}(z,y)\Big)dzds
\\
&=:I+II.
\end{split}
\end{equation*}
Here we have integrated by parts and used that the boundary terms vanish
due to $h(0)=h(1)=0$. For the term $I$, by arguing as in
(\ref{H1bound})-(\ref{pdeltaconv}) (with $s$ and $z$ in place of
$\Delta $ and $x$), we have that
\begin{equation*}
\begin{split}
&\forall s\in (0,\Delta /2]:
~~~
\sup_{z\in [0,1]}\Big|\int_{0}^{1} p_{s,b}(z,u)\varphi_{\delta }(u)du-p
_{s,b}(z,y)\Big|\xrightarrow{\delta \to 0}0,
\end{split}
\end{equation*}
showing that the $ds$-integrand in $I$ tends to $0$ pointwise. By the
heat kernel estimate (\ref{hkest}) and (\ref{eq boundedness in H^s}),
we can also bound the $ds$-integrand uniformly in $\delta $ by
\begin{equation*}
\frac{2C}{\sqrt{s}}\lVert p_{\Delta -s,b}(x,\cdot )h\rVert_{H^{1}}
\leq \frac{2C\lVert h\rVert_{H^{1}}}{\sqrt{s}}
\sup_{x\in [0,1], s\in [\Delta /2,\Delta ]}\lVert p_{s,b}(x,\cdot )
\rVert_{H^{1}}<\infty ,
\end{equation*}
where $C$ is the constant from (\ref{hkest}). Hence, we have by the
dominated convergence theorem that $|I|\xrightarrow{\delta \to 0}0$.

For $II$, we argue similarly. By Lemma
\ref{lem bounds on transition densities}, we have that
\begin{equation*}
\sup_{s\in \left[ \Delta / 2,\Delta \right] ,z\in [0,1]} \lVert \partial
_{z}p_{s,b}(z,\cdot )\rVert_{H^{2}}<\infty ,
\end{equation*}
whence (\ref{Pt-H1}) yields that
\begin{equation*}
\begin{split}
&\Big| \partial_{z}\Big(\int_{0}^{1} p_{s,b}(z,u)\varphi_{\delta }(u)du-p
_{s,b}(z,y)\Big)\Big|
\\
&
\qquad
=\Big|\int_{0}^{1} \partial_{z}p_{s,b}(z,u)\varphi_{\delta }(u)du-
\partial_{z}p_{s,b}(z,y)\Big|
\\
&
\qquad
\leq \left\|  P_{\delta ,0}\big(\partial_{z}p_{s,b}(z,\cdot )\big)-
\partial_{z}p_{s,b}(z,\cdot ) \right\| _{\infty }\xrightarrow{\delta
\to 0}0.
\end{split}
\end{equation*}
Moreover, the $ds$-integrand is bounded by (cf. Lemma
\ref{lem bounds on transition densities})
\begin{equation*}
\begin{split}
\frac{2C\|h\|_{\infty }}{\sqrt{\Delta -s}}
\sup_{s\in [\Delta /2,\Delta ],z\in [0,1]} \|p_{s,b}(z,\cdot )\|_{H
^{1}},
\end{split}
\end{equation*}
such that by dominated convergence, we have $|II|\xrightarrow{
\delta \to 0}0$.

\section{Spectral analysis of $\mathcal{L}_{b}$ and $(P_{t,b}:t\geq 0)$}\label{secspectral}
In this section, we collect some properties of the generator
$\mathcal{L}_{b}$, the differential equation related to $\mathcal{L}
_{b}$ and the transition semigroup $(P_{t,b}:t\geq 0)$ which are needed
for the proofs of Section \ref{seclocalapprox}. Although some results
can be obtained using well-known, more general theory, our proofs are
based on more or less elementary arguments, using the spectral analysis
of $\mathcal{L}_{b}$ in Section \ref{subseceigvaleigfunc}.

\subsection{Bounds on eigenvalues and eigenfunctions of
$\mathcal{L}_{b}$}\label{subseceigvaleigfunc}
The following lemma summarizes some key properties of the eigenparis
$(u_{j},\lambda_{j})$ of $\mathcal{L}_{b}$. Note that the estimate
(\ref{Halphaest}) is an improvement on the bound in Lemma 6.6 of
\cite{ghr}, and that (\ref{Halphaest}) moreover coincides with the
intuition from the eigenvalue equation $\mathcal{L}_{b}u_{j}=\lambda
_{j} u_{j}$ that ``two derivatives of $u_{j}$ correspond to one order
of growth in $\lambda_{j}$''.
\begin{lem}
\label{lem bounds on eigenfunctions}
Let $s\ge 1$ be an integer and $B>0$.
\begin{enumerate}
\item
Suppose $b\in H^{s}\cap C^{1}_{0}$. Then for all $j\geq 0$, we have
$u_{j}\in H^{s+2}$.
\item
There exist $0<C'<C<\infty $ such that for all $b\in C^{1}_{0}$ with
$\|b\|_{\infty }\le B$,
\begin{equation}
\label{lam-j-est}
\forall j\geq 0,
~~~~
\lambda_{j} \in [-Cj^{2}, -C'j^{2}].
\end{equation}
Moreover, we have $u_{0}=1, \lambda_{0}=0$.
\item
There exists $C<\infty $ such that for all $0 \le \alpha \le s+2$,
\begin{equation}
\label{Halphaest}
\forall j\geq 1: \sup_{b\in H^{s}\cap C^{1}_{0}:\|b\|_{H^{s}}\leq B}
\|u_{j}\|_{H^{\alpha }}\leq C|\lambda_{j}|^{\frac{\alpha }{2}}.
\end{equation}
In particular, we have $\|u_{j}\|_{\infty } \lesssim |\lambda_{j}|^{1/4+
\epsilon }$ for all $\epsilon >0$.
\end{enumerate}
\end{lem}
\begin{proof}
Using that $u_{j}\in \mathcal{D}\subseteq H^{2}$ and (\ref{h-mult}), we
obtain that for all $j\geq 0$, $u_{j}''=\lambda u_{j}-bu_{j}'\in H
^{1}.$ Differentiating this equation $s-1$ times and bootstrapping this
argument yields that $u_{j}^{(s+1)}\in H^{1}$.

Next, we prove (\ref{lam-j-est}) by adapting arguments from Chapter 4
of \cite{daviesspectral}. The standard Laplacian $\mathcal{L}
_{0}=\Delta $ with domain $\mathcal{D}$ is a nonpositive operator,
self-adjoint with respect to the $L^{2}$-inner product, with spectrum
$\left\{ -j^{2}\pi^{2}:j=0,1,2,...\right\} $ and associated quadratic form
\begin{equation*}
Q_{0}(f)=\langle f',f'\rangle_{L^{2}}\quad \textnormal{ for all }f
\in Dom((-\mathcal{L}_{0})^{1/2})=H^{1},
\end{equation*}
where the fact that $Dom((-\mathcal{L}_{0})^{1/2})=H^{1}$ is shown in
Chapter 7 of \cite{daviesspectral}. Similarly, using
(\ref{eqinvariantmeasure}) and integrating by parts using
$f'(0)=f'(1)=0$, we have that $\mathcal{L}_{b}$, with domain
$\mathcal{D}$, is self-adjoint with respect to the
$L^{2}(\mu_{b})$-inner product, and that for any $f\in \mathcal{D}$, the
associated quadratic form is given by
\begin{equation}
\begin{split}
\label{quad-form}
Q_{b}(f)
&=\langle -\mathcal{L}_{b} f,f\rangle_{L^{2}(\mu_{b})}=\int
_{0}^{1}f'^{2}\mu_{b}dx+\int_{0}^{1}f'f\mu_{b}'dx-\int_{0}^{1}f'fb\mu
_{b}dx
\\
&=\langle f',f'\rangle_{L^{2}(\mu_{b})}.
\end{split}
\end{equation}
For any finite-dimensional subspace $L\subseteq \mathcal{D}$, define
\begin{equation}
\label{eqdaviesfinitedimsubspace}
\begin{split}
\lambda^{(0)}(L):=\inf_{f\in L,\lVert f\rVert_{L^{2}}\leq 1}-Q_{0}(f),
\qquad
\lambda^{(b)}(L):=\inf_{f\in L,\lVert f\rVert_{L^{2}(\mu_{b})}\leq 1}-Q
_{b}(f).
\end{split}
\end{equation}
Then by Theorem 4.5.3 of \cite{daviesspectral}, the eigenvalues
of $\mathcal{L}_{0}$ and $\mathcal{L}_{b}$ are given by
\begin{equation}
\label{eqdaviesvariationalformula}
\begin{split}
\lambda_{j}^{(0)}=\sup_{L\subseteq \mathcal{D}, \dim L\leq j}
\lambda^{(0)}(L)=-j^{2}\pi^{2},
~~~~
\lambda_{j}^{(b)}=\sup_{L\subseteq \mathcal{D}, \dim L\leq j}\lambda
^{(b)}(L)
\end{split}
\end{equation}
respectively. This, combined with (\ref{quad-form}) and
(\ref{eqinvmeasbdd}), yields (\ref{lam-j-est}).

We now prove (\ref{Halphaest}). Iterating the equation $\mathcal{L}
_{b}u_{j}=\lambda_{j}u_{j}$, we have
\begin{equation*}
\begin{split}
\mathcal{\lambda }_{j}^{2}u_{j}=\mathcal{L}_{b}^{2}u_{j}
&=(u_{j}''+bu
_{j}')''+b(u_{j}''+bu_{j}')'
\\
&=u_{j}^{(4)}+b''u_{j}'+2b'u_{j}''+bu_{j}'''+bu_{j}'''+bb'u_{j}'+b
^{2}u_{j}''.
\end{split}
\end{equation*}
Note that in each summand above, except for the first one, the sum of
the orders of all derivatives is at most 3. This generalizes to
$n\geq 3$, in that there exist polynomials $P_{n,m}$ such that
\begin{equation}
\label{eqeven derivatives}
\lambda_{j}^{n}u_{j}=\mathcal{L}_{b}^{n}u_{j}=u_{j}^{(2n)}+\sum_{m=1}
^{2n-1}P_{n,m}(b,b',...,b^{(2n-2)})u_{j}^{(m)},
\end{equation}
for which one can check the following properties by induction:
\begin{enumerate}
\item
For all $n\geq 1$ and $m\leq 2n-1$, $P_{n,m}$ has degree at most $n$.
\item
The only summand in (\ref{eqeven derivatives}) with factor $b^{(2n-2)}$
is $u_{j}'b^{(2n-2)}$.
\end{enumerate}
For the odd order derivatives of $u_{j}$, there similarly exist
polynomials $\tilde{P}_{n,m}$ of degree at most $n$ such that
\begin{equation}
\label{eq odd derivatives}
\begin{split}
u_{j}^{(2n+1)}
&=\Big(\mathcal{L}_{b}^{n}u_{j}-\sum_{m=1}^{2n-1}P_{n,m}(b,b',...,b
^{(2n-2)})u_{j}^{(m)}\Big)'
\\
&= \lambda_{j}^{n} u_{j}'-\sum_{m=1}^{2n}\tilde{P}_{n,m}(b,b',...,b
^{(2n-1)})u_{j}^{(m)},
\end{split}
\end{equation}
where the only summand containing the factor $b^{(2n-1)}$ is
$u_{j}'b^{(2n-1)}$.

We now use these facts to show (\ref{Halphaest}) by an induction
argument, consisting of the base case and two induction steps.

\emph{Base Case $\alpha \leq 2$:} To show (\ref{Halphaest}) for all
$\alpha \leq 2$, it suffices to prove the case $\alpha =2$, as the case
$\alpha \in (0,2)$ then follows from $\|u_{j}\|_{L^{2}(\mu )}=1$ and
(\ref{eqsobolevinter}). We also note that the estimate for $\|u_{j}\|
_{\infty }$ then follows by the Sobolev embedding (\ref{eqsobolevemb}).
The case $\alpha =2 $ follows immediately from (\ref{norm-equiv}) and
(\ref{lam-j-est}):
\begin{equation*}
\|u_{j}\|^{2}_{H^{2}}\simeq \|\mathcal{L}_{b}u_{j}\|_{L^{2}(\mu )}
^{2}+ \|u_{j}\|_{L^{2}(\mu )}^{2}=(\lambda_{j}^{2}+1)\|u_{j}\|^{2}
_{L^{2}(\mu )}=\lambda_{j}^{2}+1\lesssim \lambda_{j}^{2},~~~ j\ge 1.
\end{equation*}

\textit{Induction step $2n\to 2n+1$:} Assume that for some integer
$n$, (\ref{Halphaest}) holds for all $\alpha \leq 2n<s+2$. Then, using
(\ref{eq odd derivatives}), the Sobolev embedding $C^{2n-2}\subseteq
H^{s}$ (note that $s\ge 2n-1$) and the induction hypothesis, we obtain
\begin{equation*}
\begin{split}
\|u_{j}^{(2n+1)}\|_{L^{2}}
&\lesssim |\lambda_{j}|^{n}\|u_{j}'\|_{L
^{2}}+\|b^{(2n-1)}\|_{L^{2}}\|u_{j}'\|_{\infty } +\|b\|^{n}_{C^{2n-2}}
\|u_{j}\|_{H^{2n}}\lesssim |\lambda_{j}|^{n+\frac{1}{2}}.
\end{split}
\end{equation*}
The non-integer case $\alpha \in (2n,2n+1)$ follows by interpolation.

\textit{Induction step $2n-1\to 2n$:} Similarly, using
(\ref{eqeven derivatives}), the embedding $C^{2n-3}\subseteq H^{s}$
(note that $s\ge 2n -2$) and the induction hypothesis, we have
\begin{equation*}
\begin{split}
\|u_{j}^{(2n)}\|_{L^{2}}\lesssim |\lambda_{j}|^{n}+\|b^{(2n-2)}\|_{L
^{2}}\|u_{j}'\|_{\infty }+\|b\|_{C^{2n-3}}^{n}\|u_{j}\|_{H^{2n-1}}
\lesssim |\lambda_{j}|^{n},
\end{split}
\end{equation*}
and the non-integer case $\alpha \in (2n-1,2n)$ again follows by
interpolation.
\end{proof}

\subsection{Characterisation of Sobolev norms in terms of
$(\lambda_{j},u_{j})$}

Using Lemma \ref{lem bounds on eigenfunctions}, we now prove that the
graph norms of the non-negative self-adjoint operators $\left( -
\mathcal{L}_{b}\right) ^{\theta }$, $\theta \in \{0,\frac{1}{2},1\}$, on
their respective domains, are equivalent to standard Sobolev norms. Let
$\ell^{2} = \ell^{2}( \mathbb{N}\cup \{0\})$ denote the usual space of
square-summable sequences. For any Banach space $(X,\|\cdot \|_{X})$ and
linear operator $T:D\to X$ with domain $D\subseteq X$, we denote the
graph norm of $T$ by
\begin{equation*}
\|x\|_{T}:= (\|x\|^{2}_{X}+\|Tx\|^{2}_{X})^{1/2}, ~~~~ x\in D.
\end{equation*}
\begin{lem}
\label{lemnormequiv}
1. Let $\theta \in [0,1]$. Then for any $f\in L^{2}$, we have
\begin{equation}
\label{dom-char}
f\in \text{Dom} \big((-\mathcal{L}_{b})^{\theta }\big)\iff \sum_{j=0}
^{\infty } (1+|\lambda_{j}|^{2\theta })|\langle f,u_{j}
\rangle_{L^{2}(\mu )}|^{2}\; <\infty
\end{equation}
and for any $f\in \text{Dom} \big((-\mathcal{L}_{b})^{\theta }\big)$,
we have
\begin{equation}
\label{spec-rep}
(-\mathcal{L}_{b})^{\theta }f=\sum_{j=1}^{\infty }(-\lambda_{j})^{
\theta }\langle f, u_{j}\rangle_{L^{2}(\mu )}u_{j}.
\end{equation}

2. There exists $0<C<\infty $ such that for any $\theta \in \{0,
\frac{1}{2},1\}$ , we have
\begin{equation}
\label{norm-equiv}
C^{-1}\|f\|_{H^{2\theta }}\leq \|f\|_{(-\mathcal{L}_{b})^{\theta }}
\leq C\|f\|_{H^{2\theta }},
~~
f\in \text{Dom} \big((-\mathcal{L}_{b})^{\theta /2}\big).
\end{equation}

3. There exists $0<C<\infty $ such that for all $f\in L^{2}$,
\begin{equation}
\label{neg-sob-equiv}
C^{-1}\|f\|_{H^{-1}} \leq \Big \| \big(\frac{\langle f,u_{j}
\rangle_{L^{2}(\mu )}}{\sqrt{1+|\lambda_{j}|}}:j\ge 0\big)\Big \|
_{\ell^{2}} \leq C\|f\|_{H^{-1}}
\end{equation}
\end{lem}
\begin{proof}
1. We first prove (\ref{dom-char}) for $\theta =1$. Define the dense
linear subspace
\begin{equation*}
D:=\bigcup_{n=0}^{\infty }\operatorname{span}\left\{  u_{j} :j=0,...,n\right\}
\subseteq L^{2}(\mu_{b}).
\end{equation*}
Then by Lemma 1.2.2 in \cite{daviesspectral}, we know that the
restriction of $\mathcal{L}_{b}$ to $D$, which we shall denote by
$\mathcal{L}_{b}^{D}$, is an essentially self-adjoint operator on
$L^{2}(\mu_{b})$. Moreover, under the unitary operator
\begin{equation*}
U:L^{2}(\mu_{b})\to \ell^{2}, \quad f\mapsto \big(\langle f,u_{j}
\rangle_{L^{2}(\mu_{b})}:j\geq 0\big),
\end{equation*}
$\mathcal{L}_{b}^{D}$ is unitarily equivalent to the essentially
self-adjoint multiplication operator $M^{D}: (a_{j}:j\ge 0)\mapsto (
\lambda_{j}a_{j}:j\ge 0)$ on $\ell^{2}$ with domain
\begin{equation*}
U(D)=\{a\in \ell^{2}: a_{j}=0 \;\text{for all}\;j \;
\text{large enough}\}.
\end{equation*}
Thus, the unique self-adjoint extentions of both operators (cf.
\cite{daviesspectral}, Theorem 1.2.7), which we denote by $
\mathcal{L}_{b}$ and $M$, are also unitarily equivalent. Hence, for all
$f\in L^{2}(\mu_{b})$,
\begin{equation*}
f\in \mathcal{D}\iff \sum_{j=0}^{\infty }(1+\lambda_{j}^{2})|\langle
f,u_{j}\rangle_{L^{2}(\mu_{b})}|^{2}<
\infty
\end{equation*}
(The above condition defines the domain of the self-adjoint extension
of $M^{D}$, see \cite{daviesspectral}, Lemma 1.3.1), which proves
(\ref{dom-char}) for $\theta =1$. To see (\ref{dom-char}) for
$\theta \in [0,1)$, we note that the fractional power $(- \mathcal{L}
_{b})^{\theta }$ is unitarily equivalent to multiplication with
$\big(|\lambda_{j}|^{\theta }:j\geq 0\big)$, and that $f\in
\text{Dom}\left( (-\mathcal{L}_{b})^{\theta }\right) $ iff
\begin{equation*}
Uf \in \text{Dom}(M^{\theta })=\big \{f\in L^{2}: \sum_{j=0}^{\infty
} (1+|\lambda_{j}|^{2\theta })|\langle f,u_{j}\rangle_{L^{2}(\mu )}|^{2}
\; <\infty \big \}.
\end{equation*}

2. We now show (\ref{norm-equiv}). For $\theta =0$, there is nothing to
prove. For $\theta =1/2$, note that by Theorem 7.2.1 in
\cite{daviesspectral} and (\ref{quad-form}), we have $\text{Dom}\left( (-
\mathcal{L}_{b})^{1/2}\right) =H^{1}$ and
\begin{equation*}
\forall f\in H^{1}:\|f\|_{\mathcal{L}_{b}^{1/2}}^{2}=\|f\|^{2}_{L^{2}(
\mu_{b})}+\langle \mathcal{L}_{b}^{1/2}f, \mathcal{L}_{b}^{1/2} f
\rangle_{L^{2}(\mu_{b})}=\|f\|^{2}_{H^{1}(\mu_{b})}.
\end{equation*}
The case $\theta =1/2$ now follows from (\ref{eqinvmeasbdd}). Finally,
let $\theta =1$. It is clear that $\|f\|_{\mathcal{L}_{b}}^{2}\lesssim
\|f\|_{H^{2}}^{2}$, so that it remains to show $\|f\|_{H^{2}}^{2}
\lesssim \|f\|_{\mathcal{L}_{b}}^{2}$. For this, we use Cauchy's
inequality with $\epsilon $ to obtain that for some $c_{1}$,
\begin{equation*}
\begin{split}
\|\mathcal{L}_{b}f\|_{L^{2}}^{2}
&= \|f''\|_{L^{2}}^{2}+2\langle f'',bf'
\rangle_{L^{2}}+\|bf'\|_{L^{2}}^{2}\geq \frac{1}{2}\|f''\|_{L^{2}}
^{2}-c_{1}\|f'\|_{L^{2}}^{2}.
\end{split}
\end{equation*}
Hence, integrating by parts and using Cauchy's inequality with
$\epsilon $ again yields that for some $c_{2}$,
\begin{equation*}
\begin{split}
\|f''\|_{L^{2}}^{2}
&\leq 2\|\mathcal{L}_{b}f\|_{L^{2}}^{2}+2c_{1}\|f'
\|_{L^{2}}^{2}\leq 2\|\mathcal{L}_{b}f\|_{L^{2}}^{2}+2c_{1}\|f\|_{L
^{2}}\|f''\|_{L^{2}}
\\[4pt]
&\leq 2\|\mathcal{L}_{b}f\|_{L^{2}}^{2}+c_{2}\|f\|_{L^{2}}^{2}+
\frac{1}{2}\|f''\|_{L^{2}}^{2},
\end{split}
\end{equation*}
proving that $\|f\|_{H^{2}}^{2}\lesssim \|f\|_{\mathcal{L}_{b}}^{2}$.

3. For any $f\in L^{2}$ and any test function $\psi \in H^{1}$, let us
write $f_{j}=\langle f, u_{j}\rangle_{L^{2}(\mu_{b})}$ and $\psi_{j}=
\langle \psi , u_{j}\rangle_{L^{2}(\mu_{b})}$, $j\geq 0$ respectively.
Then by (\ref{spec-rep})-(\ref{norm-equiv}), we have
\begin{equation}
\label{neg-sob-est}
\everymath{\displaystyle}
\begin{array}{r@{\;}l}
\|f\|_{H^{-1}}
&\simeq \sup_{\psi \in H^{1}, \|\psi \|_{H^{1}}\leq 1}
\left| \langle f,\psi \rangle_{L^{2}(\mu_{b})}\right| =
\sup_{\psi \in H^{1}, \|\psi \|_{H^{1}}\leq 1}\Big|\sum_{j=0}^{\infty
}f_{j}\psi_{j}\Big|
\\\noalign{\vspace{6pt}}
&\simeq \sup_{\psi \in L^{2}, \|\psi \|_{L^{2}}\leq 1}\Big|\sum_{j=0}
^{\infty }f_{j}\left( 1+|\lambda_{j}|\right) ^{-1/2}\psi_{j}\Big|
\\\noalign{\vspace{6pt}}
&\simeq \big \|\big(f_{j}\left( 1+|\lambda_{j}|\right) ^{-1/2}:j\ge 0
\big)\big \|_{\ell^{2}}.
\end{array}
\end{equation}
\vspace{-28pt}

\mbox{}
\end{proof}

\subsection{Basic norm estimates for the one-dimensional Neumann problem}\label{subsecneumannreg}
From the preceding Lemma, we can immediately derive some basic
properties of the (elliptic) boundary value problem
\begin{equation}
\label{eqneumannpde}
\mathcal{L}_{b}u=f
~~
\text{on}~ (0,1),
~~~~
u'(0)=u'(1)=0
\end{equation}
needed in the proof of Lemma \ref{lemregest}. Let us denote the
orthogonal complement of the first eigenfunction $u_{0}\equiv 1$ of
$\mathcal{L}_{b}$ in $L^{2}(\mu_{b})$ by
\begin{equation*}
u_{0}^{\perp }=\Big \{f\in L^{2}: \int fd\mu =0\Big \}.
\end{equation*}
\begin{lem}
\label{lemneumannregularity}
For every $f\in u_{0}^{\perp }$, there exists a unique function
$u\in \mathcal{D}\cap u_{0}^{\perp }$ such that $\mathcal{L}_{b}u=f$,
for which we use the notation $u=\mathcal{L}_{b}^{-1}f$. Moreover, for
every $B>0$ there exists $C<\infty $ such that for all $b\in C^{1}
_{0}$ with $\|b\|_{\infty }\leq B$ and $f\in u_{0}^{\perp }$,
\begin{equation}
\label{ell-norm-est}
\|u\|_{H^{s}}\leq C\|f\|_{H^{s-2} }\quad \textnormal{ for } s\in \{0,1,2
\}.
\end{equation}
\end{lem}
\begin{proof}
It follows immediately from the domain characterisation (\ref{dom-char})
and the spectral representation (\ref{Lb-spectral}) that $\mathcal{L}
_{b}$ is a one-to-one map from $\mathcal{D}\cap u_{0}^{\perp }$ to
$L^{2}\cap u_{0}^{\perp }$, and that $\mathcal{L}_{b}^{-1}$ is unitarily
equivalent to multiplication by $(\lambda_{j}^{-1}\mathbh{1}_{j
\geq 1}:j\ge 0)$ in the spectral domain, so that the $L^{2}\to L^{2}$
norm of $\mathcal{L}_{b}^{-1}$ is finite. Hence, for $s=2$, the estimate
(\ref{ell-norm-est}) follows from (\ref{norm-equiv}):
\begin{equation*}
\begin{split}
\left\| \mathcal{L}_{b}^{-1}f\right\| ^{2}_{H^{2}}\simeq \left\| \mathcal{L}
_{b}\mathcal{L}_{b}^{-1}f\right\| _{L_{2}}^{2}+\|\mathcal{L}_{b}^{-1}f
\|_{L^{2}}^{2}\simeq \|f\|_{L_{2}}^{2}.
\end{split}
\end{equation*}
The case $s=0$ is obtained by duality. Using that $\mathcal{L}_{b}
^{-1}$ is self-adjoint on $u_{0}^{\perp }$ and the previous case $s=2$,
we have that
\begin{equation*}
\begin{split}
\|\mathcal{L}_{b}^{-1}f\|_{L^{2}(\mu_{b})}
&=
\sup_{\phi \in u_{0}^{\perp }, \|\phi \|_{L^{2}}\leq 1}\Big|\int_{0}
^{1} \mathcal{L}_{b}^{-1}f\phi d\mu \Big|=
\sup_{\phi \in u_{0}^{\perp }, \|\phi \|_{L^{2}}\leq 1}\Big|\int_{0}
^{1} f\mathcal{L}_{b}^{-1}\phi d\mu \Big|
\\
&\lesssim \|f\|_{H^{-2}}.
\end{split}
\end{equation*}
Finally, for $s=1$, Lemma \ref{lemnormequiv} implies that
\begin{equation*}
\left\| \mathcal{L}_{b}^{-1}f\right\| _{H^{1}}^{2}
 \simeq \sum_{j=1}
^{\infty }(1+|\lambda_{j}|)\big|\frac{\langle f,u_{j}
\rangle_{L^{2}(\mu_{b})}}{\lambda_{j}}\big|^{2}\lesssim \sum_{j=1}
^{\infty }\frac{|\langle f,u_{j}\rangle_{L^{2}(\mu_{b})}|^{2}}{1+|
\lambda_{j}|}\lesssim \|f\|_{H^{-1}}^{2}.
             \qedhere
\end{equation*}
\end{proof}

\subsection{Estimates on $p_{t,b}(\cdot ,\cdot )$ and $P_{t,b}$}\label{subsecfixedtime}
Using Lemmata \ref{lem bounds on eigenfunctions} and
\ref{lemnormequiv}, we now collect some basic (partially well-known)
results about the Lebesgue transition densities $p_{t,b}(\cdot ,
\cdot )$ (Lemma \ref{lem bounds on transition densities}) and the
semigroup $P_{t,b}$ (Lemma \ref{lemPtsmoothing}). Recall that they were
defined in (\ref{eqtransitiondensity}) and (\ref{eqsemigroup}).
\begin{lem}
\label{lem bounds on transition densities}
Let $s\ge 1$ be an integer, $t_{0}>0$ and $B>0$. Then we have the
following.
\begin{enumerate}
\item
There exist constants $0<C<C'<\infty $ such that for all $t\ge t_{0}$,
$b\in C^{1}_{0}$ with $\|b\|_{C^{1}}\leq B$ and $x,y\in [0,1]$,
\begin{equation}
\label{eqtransdensitybounded}
C\leq p_{t,b}(x,y)\leq C'.
\end{equation}

\item
There exists $C<\infty $ such that for all $t\in (0,1]$ and
$b\in C^{1}_{0}$ with $\|b\|_{\infty }\leq B$,
\begin{equation}
\label{hkest}
\|p_{t,b}(x,y)\|_{\infty }\leq Ct^{-\frac{1}{2}},
~~
x,y\in [0,1].
\end{equation}

\item
For each $n\leq s+2,\; m\leq s$ and $n',m'\leq s+1,$
\begin{equation}
\label{eq boundedness in H^s}
\begin{split}
&\sup_{t\ge t_{0}}\sup_{y\in [0,1]}
\sup_{b\in C^{1}_{0}\cap H^{s}:\|b\|_{H^{s}}\leq B}\|\partial_{x}^{n}
\partial_{y}^{m}p_{t,b}(\cdot ,y)\|_{L^{2}}<\infty
\\
&\sup_{t\ge t_{0}}\sup_{x\in [0,1]}
\sup_{b\in C^{1}_{0}\cap H^{s}:\|b\|_{H^{s}}\leq B} \|\partial_{x}
^{n'}\partial_{y}^{m'}p_{t,b}(x,\cdot )\|_{L^{2}}<\infty .
\end{split}
\end{equation}
\end{enumerate}
\end{lem}
\begin{proof}
For a proof of (\ref{eqtransdensitybounded}), we refer to Proposition
9 in \cite{ns} and for a proof of (\ref{hkest}), we refer to
Theorem 2.12 in \cite{chorowski}. Let us now prove the first part
of (\ref{eq boundedness in H^s}); the second is obtained analogously.
Let $n\leq s+2,\; m\leq s$. Then (\ref{eqinvariantmeasure}) yields that
\begin{equation*}
\sup_{\|b\|_{H^{s}}\leq B}\|\mu_{b}\|_{H^{s+1}}<\infty .
\end{equation*}
Using the multiplicative inequality (\ref{h-mult}), the spectral
decomposition (\ref{eqspectral}) and Lemma
\ref{lem bounds on eigenfunctions}, we have
\begin{equation*}
\begin{split}
&\lVert \partial_{x}^{n}\partial_{y}^{m}p_{t,b}(\cdot ,y)\rVert_{L
^{2}} \leq \sum_{j=0}^{\infty }e^{t\lambda_{j}}\lVert u_{j}^{(n)}
\rVert_{L^{2}}|\left(  u_{j}\mu_{b}\right) ^{(m)}(y)|
\\[-3pt]
&
~~~
\leq \sum_{j=0}^{\infty }e^{t_{0}\lambda_{j}}\lVert u_{j}^{(n)}
\rVert_{L^{2}}\lVert \left(  u_{j}\mu_{b}\right) ^{(m)}\rVert_{\infty }
\lesssim \sum_{j=0}^{\infty }e^{t_{0}\lambda_{j}}\lVert u_{j}
\rVert_{H^{s+2}}\lVert u_{j}\rVert_{H^{s+1}}\lVert \mu_{b}\rVert_{H
^{s+1}}
\\[-3pt]
&
~~~
\lesssim \sum_{j=0}^{\infty }e^{-cj^{2}} |\lambda_{j}|^{\frac{s+2}{2}+
\frac{s+1}{2}}\lesssim \sum_{j=0}^{\infty }e^{-cj^{2}} j^{2s+3} <
\infty ,
\end{split}
\end{equation*}
where Lemma \ref{lem bounds on eigenfunctions} implies that the
constants above are uniform in $\|b\|_{H^{s}}\leq B$.
\end{proof}

Finally, we collect some properties of $(P_{t,b}:t\geq 0)$.
\begin{lem}
\label{lemPtsmoothing}
Let $B>0$. The following holds.
\begin{enumerate}
\item
For all $b\in C^{1}_{0}$, $p\in [1,\infty ]$ and $f\in L^{p}$, we have
$\|P_{t,b}f\|_{L^{p}(\mu )}\leq \|f\|_{L^{p}(\mu )}$.
\item
For every $\epsilon >0$, there exists $C<\infty $ such that for all
$b\in C^{1}_{0}$ with $\|b\|_{\infty }\leq B$, $f\in H^{1}$ and
$t> 0$,
\begin{equation}
\label{Pt-H1}
\|P_{t,b}f-f\|_{L^{2}}\leq Ct^{1/2} \|f\|_{H^{1}}
~~
\text{and}
~~
\|P_{t,b}f-f\|_{\infty }\leq Ct^{1/4-\epsilon } \|f\|_{H^{1}}.
\end{equation}
In particular, we have that $H^{1}\subseteq \mathcal{D}(1/2)$, with
$D(1/2)$ defined by (\ref{eqDalpha}).
\item
Let $s\ge 1$ be an integer. Then for all $t>0$, $b\in H^{s}\cap C^{1}
_{0}$ with $\|b\|_{H^{s}}\leq B$ and $f\in L^{2}$, we have $P_{t,b}f
\in H^{s+2}$. Moreover, there exists $C<\infty $ such that for all such
$t,b,f$ and all $\alpha \le s+2$,
\begin{equation}
\label{Pt-H1H-1}
\|P_{t,b}f\|_{H^{\alpha }}\leq C(1+t^{-\frac{\alpha }{2}-\frac{3}{4}})
\|f\|_{H^{-1}}.
\end{equation}
\end{enumerate}
\end{lem}

\begin{proof}
1. For the case $p=1$, we have by Fubini's theorem that
\begin{align*}
\int_{0}^{1}\Big|\int_{0}^{1} p_{t,b}(x,z)f(z)dz\Big|d\mu (x)&\leq \int
_{0}^{1}\int_{0}^{1} p_{t,b}(x,z)d\mu (x)|f(z)|dz\\&=\int_{0}^{1} |f(z)|d
\mu (z).
\end{align*}
For the case $p=\infty $, we observe that for all $x\in [0,1]$
\begin{equation*}
|P_{t,b}f(x)|\leq \|f\|_{\infty }\int p_{t,b}(x,z)dz=\|f\|_{\infty }.
\end{equation*}
The case $p\in (1,\infty )$ follows by the Riesz-Thorin interpolation
theorem.

2. To prove the first part of (\ref{Pt-H1}), let $f\in H^{1}=
\text{Dom}\left( (-\mathcal{L}_{b})^{1/2}\right) $. By the
$1/2$-H\"{o}lder continuity of $x\mapsto e^{x}$ on $ (-\infty ,0]$ and
Lemma \ref{lemnormequiv}, we have that for all $t\geq 0$,
\begin{equation*}
\begin{split}
\|P_{t,b}f-f\|_{L^{2}(\mu_{b})}^{2}
&=\sum_{j=1}^{\infty }\left( e^{
\lambda_{j}t}-1\right) ^{2}|\langle f,u_{j}\rangle_{L^{2}(\mu_{b})}|^{2}
\\
&\lesssim t\sum_{j=1}^{\infty }|\lambda_{j}||\langle f,u_{j}
\rangle_{L^{2}(\mu_{b})}|^{2}\lesssim t\|f\|_{H^{1}}^{2}.
\end{split}
\end{equation*}
The second estimate in (\ref{Pt-H1}) now follows from the $H^{1}
\to H^{1}$ boundedness of $P_{t,b}$, the embedding (\ref{eqsobolevemb}),
the interpolation inequality (\ref{eqsobolevinter}) and the first part
of (\ref{Pt-H1}). Indeed, we have for any $\varepsilon >0$ that
\begin{equation*}
\|P_{t,b}f-f\|_{\infty }\lesssim \|P_{t,b}f-f\|^{(1-4\varepsilon )/2}
_{L^{2}}\|P_{t,b}f-f\|^{(1+4\varepsilon )/2}_{H^{1}}\lesssim t^{1/4-
\varepsilon }\|f\|_{H^{1}}.
\end{equation*}

3. By Lemma \ref{lem bounds on eigenfunctions}, we have that
$u_{j}\in H^{s+2}$ for all $j\geq 0$. Using the spectral representation
(\ref{Pt-spectral}), Lemma \ref{lem bounds on eigenfunctions}, Lemma
\ref{lemnormequiv} and Cauchy-Schwarz, we have
\begin{equation*}
\everymath{\displaystyle}
\begin{array}[b]{r@{\;}l}
\|P_{t,b}f\|_{H^{\alpha }}
&\lesssim \sum_{j=0}^{\infty }e^{\lambda
_{j}t}\|u_{j}\|_{H^{\alpha }}(1+|\lambda_{j}|)^{1/2}\frac{|\langle f,u
_{j}\rangle_{L^{2}(\mu_{b})}|}{(1+|\lambda_{j}|)^{1/2}}
\\\noalign{\vspace{4pt}}
&\lesssim \Big(\sum_{j=0}^{\infty }e^{2\lambda_{j}t}(1+|\lambda_{j}|)^{
\alpha +1}\Big)^{1/2}\|f\|_{H^{-1}}
\\\noalign{\vspace{4pt}}
&\lesssim \Big(1+\int_{0}^{\infty }e^{-2cx^{2}t}x^{2(\alpha +1)}dx
\Big)^{1/2}\|f\|_{H^{-1}}
\\                      \noalign{\vspace{4pt}}
&\lesssim \Big(1+t^{-\alpha -1-\frac{1}{2}}\Big)^{1/2}\|f\|_{H^{-1}}
\\                                            \noalign{\vspace{4pt}}
&\lesssim \big(1+t^{-\frac{\alpha }{2}-\frac{3}{4}}\big)\|f\|_{H^{-1}}.
\end{array}       \qedhere
\end{equation*}
\end{proof}

%%% Backmatter %%%

%Appendix
%\begin{appendix}
%\section{}\label{app1}
%\end{appendix}

%Acknowledgments
 \section*{Acknowledgments} % Acknowledgments/Acknowledgment - following manuscript
I am very grateful to Richard Nickl for suggesting to pursue this
project, for repeatedly proofreading the manuscript and for sharing with
me some of the important ideas in this paper, including the one of
applying a PDE approach. I would like to thank the Cantab Capital
Institute for the Mathematics of Information, as well as Richard Nickl's
ERC grant No. 647812 (UQMSI), for supporting my PhD.

%Supplementary material
%\begin{supplement}
%\stitle{}
%\slink[doi]{10.1214/XX-EJSXXXXSUPP}
%\sdatatype{} %.pdf, .zip
%\sfilename{}
%\sdescription{}
%\end{supplement}

%References

\end{document}